\documentclass[preprint,12pt]{elsarticle}
    \usepackage{amssymb}
     \usepackage{amsthm}
    \usepackage{amsfonts}
    \usepackage{amsmath,amssymb}
    \usepackage[applemac]{inputenc}
    \usepackage{color}
    \usepackage{color, colortbl}
    \usepackage{amsmath}
    \usepackage{amssymb}
    \usepackage{bbm}
    \usepackage{graphicx,floatrow}
    \usepackage{subfigure}
\usepackage{booktabs}
\usepackage[all]{xy}
\usepackage{mathrsfs,bbm}
    \usepackage{tikz}
    \usepackage{geometry}
    \usetikzlibrary{arrows}
    \usepackage{amsfonts}
    \usepackage{mathrsfs}
    \usepackage{amsmath,amssymb}
    \usepackage[applemac]{inputenc}
    \usepackage{color}
    \usepackage{amsmath}
    \usepackage{amssymb}
    \usepackage{graphicx}
    \usepackage{tikz}
    \newtheorem{theorem}{Theorem}[section]

    \newtheorem{proposition}[theorem]{Proposition}
    \newtheorem{corollary}[theorem]{Corollary}
    \newtheorem{definition}[theorem]{Definition}
    
    \newtheorem{example}[theorem]{Example}

    \newtheorem{remark}[theorem]{Remark}

\begin{document}
\newcommand{\A}{\mathbb{A}}
\newcommand{\Q}{\mathbb{Q}}
\newcommand{\N}{\mathbb{N}}
\newcommand{\Z}{\mathbb{Z}}
\newcommand{\R}{\mathbb{R}}
\newcommand{\M}{\mathbb{M}}
\newcommand{\tx}[1]{\quad\mbox{#1}\quad}
\newcommand{\Aut}{\mathrm{Aut}}
\newcommand{\Inn}{\mathrm{Inn}}
\newcommand{\Sym}{\mathrm{Sym}}
\newcommand{\LSec}{\mathrm{LSec}}
\newcommand{\RSec}{\mathrm{RSec}}

\begin{frontmatter}
\title{{\bf{  Insights on Stochastic Dynamics for Transmission of Monkeypox: Biological and Probabilistic Behaviour}}\tnoteref{label1}}\tnotetext[label1]{}
\author[swat]{Ghaus ur Rahman}
\ead{dr.ghaus@uswat.edu.pk}
\author[nor,kpi]{Olena Tymoshenko}
\ead{olenaty@math.uio.no}
\author[nor]{Giulia Di Nunno}
\ead{giulian@math.uio.no}
%\cortext[cor1]{Corresponding author.}

%\author[sxu]{Qaisar Badshah}
%\ead{qaisarbadshah859@gmail.com}
\address[swat]{``Department of Mathematics and Statistics, University of Swat, Pakistan"}
\address[nor]{``Department of Mathematics, University of Oslo, Norway"}
\address[kpi]{``Department of Mathematical Analysis  and Probability Theory, NTUU Igor Sikorsky Kyiv Polytechnic Institute, Ukraine"}
%\address[sxu]{``Elementary \& Secondary Education Department Khyber Pakhtunkhwa, Pakistan"}
\begin{abstract}
 The transmission of monkeypox is studied using a  stochastic model taking into account the  biological aspects, the contact mechanisms and the demographic factors together with the intrinsic uncertainties.   Our results provide insight into the interaction between stochasticity and biological elements in the dynamics of monkeypox transmission. The rigorous mathematical analysis determines  threshold parameters for disease persistence. For the proposed model, the existence of a unique global almost sure non-negative solution is proven. Conditions leading to disease extinction are established. Asymptotic properties of the model are investigated such as the speed of transmission. 

\end{abstract}
\begin{keyword}
 Monkeypox Virus;~ Stochastic Model;~Extinction; Long-Time Behaviour\\
%\emph{Mathematics subject classification:} 34D20;~47A35;~60H10
\end{keyword}
\end{frontmatter}
\section{Introduction}
\noindent The modeling of disease transmission heavily relies on stochastic differential equations (SDEs), which offer a more realistic framework for comprehending the inherent uncertainty involved in the innate variety within a population. The authors of \cite{m1} provided a comprehensive framework for understanding the modeling of infectious diseases, blending both theoretical and practical aspects of epidemiology, while \cite{m2} offers a deep dive into the principles of epidemiological modeling and control strategies. 
Ecological models have been studied using time series analysis \cite{mid1} and tools of stochastic calculus \cite{,mid2,Agarwal,ghaus3}, also see the references therein.

Monkeypox is an extremely contagious zoonotic disease that results from infection by the monkeypox virus, a member of the Orthopoxvirus genus within the Poxviridae family \cite{Durski,Jezek}. Its impact is largely felt in rural regions situated close to tropical rainforests in central and West Africa. However, there is also an increasing number of reported cases in UK and in Europe.  
Rodents are responsible for the spread and transmission of the zoonotic viral disease monkeypox. Humans can either get monkeypox directly or indirectly through rodents (mostly rats) and primates wildlife is major source of this disease.   

Furthermore, it is commonly transmitted between individuals through various modes, such as respiratory droplets, contact with contaminated bodily fluids or objects, as well as through the skin lesions of infected individuals \cite{Alakunle}. After the successful eradication of smallpox, the incidence of the orthopox virus known as monkeypox has surged, making it the most commonly encountered variant \cite{Kantele}.  
The comprehension of the ecological dynamics of rodents is important for controlling the propagation of the virus by the decreasing rodent-human interaction \cite{mid4,mid5}. 
%%%%%%%%%%%%%%%%%%%%%%
Among various methods in modeling, the compartmental approach provides an invaluable framework for describing and forecasting the dynamics of infectious diseases, like monkeypox. General theory of the compartmental approach can be found in e.g., \cite{Cas,LaSalle,kam, Fima}, and also references therein. 

In the literature, we find some papers modelling monkeyfox transmission, with deterministic perspective. As illustration, \cite{alt1, alt2, p1, p2} explore the matter explaining  
the integer order formulation, as well as the their extensions to the fractional order system by employing $\Psi$-Hilfer fractional derivative and with Mittag-Leffler kernel. 

We also use compartmental modeling.  In comparison to the existing literature on dynamics of monkeypox, our paper contributes to explore complex aspects of the diseases including the uncertainty in the model. Indeed, an epidemic's behaviour is unpredictable due to external environmental disturbances that impact its propagation. 
%%%%Alpha+%%%%%%%%%%%%%%%
We propose a compartmental model  where we have two populations.  
The human population is composed of four distinct categories: susceptible individuals $S_h$, infected individuals $I_h$, isolated individuals $Q_h$, and recovered individuals $R_h$. On the other hand, the rodent population can be divided into two groups: susceptible rodents $S_r$ and infected rodents $I_r$.
\begin{figure}[!ht]\label{fig}
\center{\includegraphics[scale=0.35]{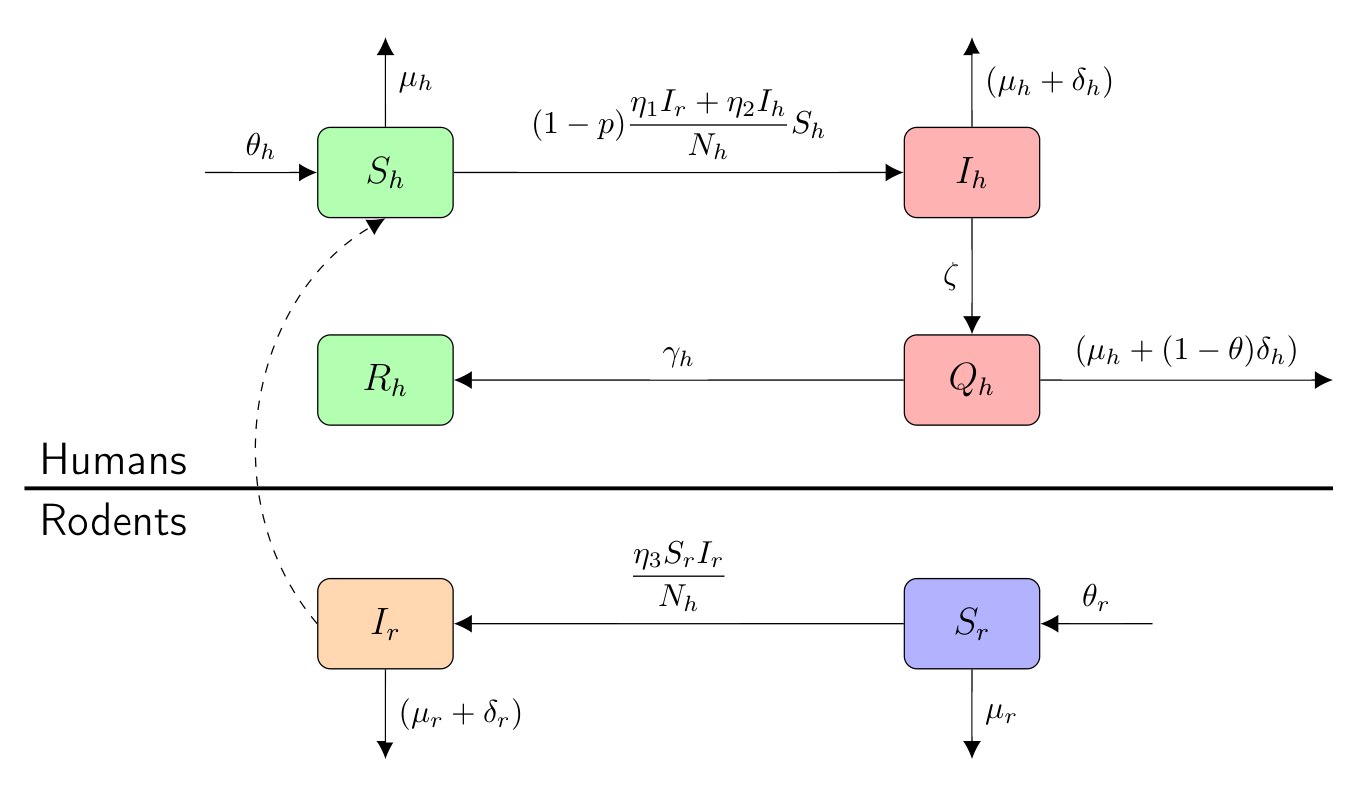}}
\caption{The schematic diagram of Monkey Pox virus transition in the deterministic case}
\end{figure}

%%%%%%%+Beta%%%%%%%%%%%%%

Starting from the deterministic model (see Figure \ref{fig}), to be able to account for the uncertainty, we add  stochastic noises modelled by Brownian motions. These act as a perturbation of the system mostly due to the effects of the uncertainty of death and birth rates as well as the imprecise count of individuals. 
%%%%%+delt
Indeed, we consider the following stochastic compartment model for monkeypox: $(S_h~,I_h,~Q_h,~R_h,)$ related to  human population and $(S_r,~I_r)$ for rodents: 
\begin{eqnarray}\label{STD1.1}
dS_h(t)&=& \left(\theta_h-(1-p)\cdot\frac{\eta_1I_r(t)+\eta_2I_h(t)}{N_h(t)}\cdot S_h(t)-\mu_hS_h(t)\right)dt\cr\cr&&-(1-p)\cdot \frac{\sigma_1I_r(t)dB_1(t)+\sigma_2I_h(t)dB_2(t)}{N_h(t)}\cdot S_h(t)+\sigma_3S_h(t)dB_3(t);\cr\cr
dI_h(t)&=&\left((1-p)\cdot \frac{\eta_1I_r(t)+\eta_2I_h(t)}{N_h(t)}\cdot S_h(t)-(\mu_h+\delta_h+\zeta)I_h(t)\right)dt\cr\cr&&+(1-p)\cdot\frac{\sigma_2S_h(t) }{N_h(t)}\cdot I_h(t)dB_2(t)+\sigma_4I_h(t)dB_4(t);\cr\cr
dQ_h(t)&=&\left(\zeta I_h(t)-(\mu_h+\gamma_h+(1-\theta)\delta_h)Q_h(t)\right)dt+\sigma_5Q_h(t)dB_5(t); \cr\cr
dR_h(t)&=&\left(\gamma_hQ_h(t)-\mu_hR_h(t)\right)dt+\sigma_6R_h(t)dB_6(t); \cr\cr
dS_r(t)&=&\left(\theta_r-\frac{\eta_3S_r(t)I_r(t)}{N_r(t)}-\mu_rS_r(t)\right)dt+\sigma_7S_r(t)dB_7(t); \cr\cr
dI_r(t)&=&\left(\frac{\eta_3S_r(t)I_r(t)}{N_r(t)}-(\mu_r+\delta_r)I_r(t)\right)dt+\sigma_8I_r(t)dB_8(t),
\end{eqnarray}
where 
$(S_h(0),I_h(0),Q_h(0),R_h(0),S_r(0)),I_r(0) \in  \mathbb{R}^6_+$ is initial value, $B_i(t)$, $i=1, ..., 8,$ are mutually independent standard  Brownian motions defined on a filtered probabilistic space $(\Omega, \mathbb{F}, \{{\cal F}_t\}_{t\geq0},\mathrm{P})$ such that ${\cal{F}}_0$ contains the null sets, and $\sigma_i, i=1, 2,..., 8,$ are their intensities (or volatilises).
We also denote the involved human population by
$N_h:=S_h+I_h+Q_h+R_h$, and the rodent population by $N_r:=S_r+I_r.$
 Note that in this paper, we  only  consider the case when  
\begin{equation}\label{HR}
    N_r(t)\leq k(t)N_h(t),
\end{equation}
where $ k(t), ~t\geq 0,$ is some continuous positive function with $\bar{k}=\underset{t\geq0}\sup\, k(t).$
Note in particular that the rate of contacts, $\eta_1$ and  $\eta_2$, are the basic parameters that are responsible for monkeypox disease's transmission. 
\noindent The susceptible human class  $S_h$ is generated by a constant recruitment rate $\theta_h$ and reduced by acquiring infection after interaction with infectious humans and rodents. The force of infection is $\eta_1 I_r+\eta_2 I_h$, where $\eta_1$ and $\eta_2$ are the effective contact rates (capable of transmitting infection) corresponding to infectious rodents and humans, respectively. The mortality rate of each class within the human population decreases over time due to natural deaths occurring at a rate of $\mu_h$. Meanwhile, the death rate caused by monkeypox in the infected and isolated human compartments is indicated by $\delta_h$. The infectious human population is isolated and joins the $Q_h$ class at the transmission rate $\zeta$. The recovery rate of the isolated/quarantine human population is denoted by $\gamma_h$.  The value $p$, with $0\leq p\leq 1$, measures the effectiveness of enlightenment campaign,  and $\theta$, with $0 \leq \theta \leq 1$ is the effectiveness of quarantine and
treatment. It is assumed that the death in $Q_h$ due to disease is influenced by the effectiveness of treatment, hence it is $(1 - \theta) \delta_h$ .

In the rodent population, $\theta_r$ shows the recruitment rate while the natural death rate  is $\mu_r$. The infection rate of the virus from an infected rodent to another rodent is denoted as $\eta_3 I_r$, with $\eta_3$ representing the transmission rate. The population of each class of rodents decreases naturally due to a death rate denoted as $\mu_r$, while the death rate caused by monkeypox infection in the infected rodent group is denoted as $\delta_r$. Since  wild rodents do not generically have access to treatment, we  assume that  they do not recover from the disease. A summary of the model elements is presented in Table 1.
%%%%%%%%%%%%%%%%%%%%%%

\begin{table}[H]
  \centering
  \begin{tabular}{|c|c|}
    \hline
    Model's Elements & Description of Elements  \\
    \hline
$N_h$ & Humans population  \\
$S_h$ & Humans who are susceptible  \\
$I_h$  & Humans who are infected\\
$Q_h$   & Infected Humans who are quarantine\\
$R_h$   &Humans who are Recovered\\
$N_r$ & Rodents population\\
$S_r$   &Rodents who are susceptible\\
$I_r$   & Rodents who are infected \\
$\theta_h$   &Recruitment Rate of Humans\\
$p$   &Effectiveness Public Enlightenment
Campaign\\
$\eta_1$   &Rate of Contact Between Rodents and Humans\\
$\eta_2$   &Rate of Contact Between Humans and Humans\\
$\eta_3$   &Rate of Contact Between Rodents and Rodents\\
$\mu_h$   &The rate at which humans naturally pass away\\
$\delta_h$ &Mortality rate caused by diseases in humans\\
$\zeta$ & Progression Rate from Infected to
Quarantine\\
$\gamma_h$ & Recovery Rate of Humans\\
$\theta$ &Effectiveness of Quarantine and Treatment\\

$\theta_r$ &Rodents growth rate from birth to adult stage\\
$\mu_r$ &The rate at which rodents naturally pass away\\
$\delta_r$ & Mortality rate caused by diseases in rodents\\
    \hline
  \end{tabular}
  \caption{Descriptions of the elements}\label{Table 1}
\end{table}
One described the stochastic model for monkeypox transmission, we proceed with a complete analysis of the existence and positivity of the solution to the stochastic dynamic system (\ref{STD1.1}) in Section 2. 
We study stochastic ultimate boundedness and permanence of the solution in Section 3. Conditions for the disease extinction are detailed in Section 4, where we also provide illustration of the obtained results with numerical examples realized with Python. Section 5 presents an extended model in which the parameters are time varying (non-autonomous system). With this, we study the long-time asymptotic behaviour of the model and give results  on the speed of the transmission. 

\section{Global Existence and Positivity of the Solution}
\noindent  In studying dynamic behavior, it is crucial to establish whether a solution exists globally, i.e., there is no explosion in a finite time. In addition, in the context of epidemic population dynamics, it is also important to assess whether the solutions are non-negative. 

A necessary condition for a stochastic differential equation to have a unique  global solution  for any given initial value, is that the coefficients have the linear growth, and are locally Lipschitz. (e.g. see \cite{Mao}, \cite{Arnold}, \cite{GS}).
 However, although the coefficients of system (\ref{STD1.1}) do not satisfy  linear growth condition due to the bilinear incidence, they are locally Lipschitz continuous. Nonetheless, the following result demonstrates that the solution  of system (\ref{STD1.1}) is exists  globally, it is unique, and non-negative.

\begin{theorem}\label{Global}
The system represented by equation (\ref{STD1.1}) has a unique solution, 
$$\mathcal{X}(t)=(S_h(t),I_h(t),Q_h(t),R_h(t),S_r(t), I_r(t)),$$ on $t \geq 0$ given an initial condition of $\mathcal{X}(0)\in \mathbb{R}^6_+$, and this solution remains in the  space $\mathbb{R}^6_+$ with a probability of 1, namely, $\mathcal{X}(t)\in \mathbb{R}^6_+$ for all $t\geq 0$ almost surely (a.s.).
\end{theorem}
\begin{proof} Since the coefficients of the equation are locally Lipschitz continuous for any given initial value $\mathcal{X}(0)\in \mathbb{R}^6_+$, by the theory of stochastic differential equations there is a unique
local solution $\mathcal{X}(t)$ on $t \in [0,\tau_e)$, where $\tau_e$ is the explosion time (see \cite{Arnold}).  To prove that this solution is global, we have to show that that $\tau_{e}=\infty$ a.s. 

This can be done by setting a sufficiently large value for $m_0\geq 0$, such that $\mathcal{X}(0) \in \left[\frac{1}{m_0}, m_0\right]^{\times 6}$. Then we define a stopping time for  each integer $m \geq m_0$:
  \begin{eqnarray*}\tau_m=&\inf&\left\{t \in [0, \tau_e): S_h(t)\notin\left(\frac{1}{m}, m\right)\lor I_h(t)\notin\left(\frac{1}{m}, m\right)\lor  Q_h(t)\notin\left(\frac{1}{m}, m\right)\lor \right.\cr\cr&& \left.R_h(t)\notin\left(\frac{1}{m}, m\right)\lor  S_r(t)\notin\left(\frac{1}{m}, m\right)\lor  I_r(t) \notin\left(\frac{1}{m}, m\right) \right\}. \end{eqnarray*} By definition, $\tau_m$ is increasing as  $m\to \infty$, and we set $\tau_\infty=\underset{m \to \infty}{\lim} \tau_m$, then $$\tau_\infty \leq \tau_e \,\,\textnormal{ a.s.}$$
 To complete the proof, we have to show that $\tau_ \infty=\infty$ a.s. If this statement is false, then we can find a pair of constants $T > 0$ and $\varepsilon \in (0,1)$ such that $P\{\tau_\infty \leq T\} > \varepsilon$. Hence, there is an integer $m_1 \geq m_0$ such that
  \begin{eqnarray}\label{eq3.1}
 \mathrm{P}\{\tau_m \leq T\} \geq \varepsilon \,\, \textnormal{for all} \,\,\, m \geq m_1.
 \end{eqnarray}

\noindent With this we consider the following  function $\mathcal{V}:\mathbb{R}^6_+\to \mathbb{R}_+$:
\begin{eqnarray*}
\mathcal{V}(S_h,I_h,Q_h,R_h,S_r,I_r)&=&(S_h-1-\log
S_h)+(I_h-1-\log I_h)\cr\cr&&+(Q_h-1-\log
Q_h)+(R_h-1-\log R_h)\cr\cr&&+(S_r-1-\log S_r)+(I_r-1-\log I_r).
\end{eqnarray*}
This function is non-negative, since $x-1-\log x\geq0$ for all $x>0$.

\noindent By It\^o's formula, it holds that
\begin{eqnarray}\label{leeq3.1}
&&d\mathcal{V}(\mathcal{X}(t))= L\mathcal{V}(\mathcal{X}(t))dt\cr\cr&&+\frac{(1-p)\cdot(\sigma_1I_r(t)S_h(t)dB_1(t)+\sigma_2I_h(t)S_h(t)dB_2(t))}{N_h(t)}\cdot\left(\frac{1}{S_h(t)}-\frac{1}{I_h(t)}\right)\cr\cr&&+(S_h(t)- 1)\sigma_3dB_3(t)
+(I_h(t)-1)\sigma_4dB_4(t)+(Q_h(t)-1)\sigma_5dB_5(t)\cr\cr&&+(R_h(t)-1)\sigma_6dB_6(t)+(S_r(t)-1)\sigma_7dB_7(t)+(I_r(t)-1)\sigma_8dB_8(t),
\end{eqnarray}
where $L\mathcal{V}:\mathbb{R}^6_+\to \mathbb{R}_+$ is defined by
\begin{eqnarray*}
&&L\mathcal{V}(\mathcal{X}(t))=\theta_h+\theta_r+4\mu_h+\delta_h+\zeta+2\mu_r+\delta_r+\gamma_h+(1-\theta)\delta_h\cr\cr&&+(1-p)\cdot\frac{(\eta_1I_r(t)+\eta_2I_h(t))}{N_h(t)}+\frac{\eta_3I_r(t)}{N_r(t)}-\mu_hS_h(t)-(\mu_h+\delta_h)I_h(t)\cr\cr&&-(\mu_h+(1-\theta)\delta_h)Q_h(t)-\mu_hR_h(t)-\mu_rS_r(t)
-(\mu_r+\delta_r)I_r(t)-\frac{\theta_h}{S_h(t)}\cr\cr&&-(1-p)\cdot\frac{(\eta_1I_r(t)+\eta_2I_h(t))S_h(t)}{N_h(t)I_h(t)}-\frac{ \zeta I_h(t)}{Q_h(t)}-\frac{\gamma_hQ_h(t)}{R_h(t)}-\frac{\theta_r}{S_r(t)}-\frac{\eta_3S_r(t)}{N_r(t)}\cr\cr&&+(1-p)^2\cdot\frac{\sigma^2_1I^2_r(t)+\sigma^2_2I^2_h(t)}{2N^2_h(t)}+(1-p)^2\cdot\frac{\sigma^2_2S^2_h(t)}{2N^2_h(t)}\cr\cr&&+\frac{\sigma^2_3}{2}+\frac{\sigma^2_4}{2}+\frac{\sigma^2_5}{2}+\frac{\sigma^2_6}{2}+\frac{\sigma^2_7}{2}+\frac{\sigma^2_8}{2}.
\end{eqnarray*}
Following  assumption (\ref{HR}) we have that $\dfrac{I_r}{N_h}\leq \bar{k}$ and $S_h< N_h$, $I_r< N_r$, $S_r< N_r$,  $I_h< N_h$, then we can see that
\begin{eqnarray}\label{leq3.2}
&L\mathcal{V}(\mathcal{X}(t))&\leq K,
\end{eqnarray}
with the constant $K$ given by \begin{eqnarray*}
K&=&\theta_h+\theta_r+4\mu_h+\delta_h+\zeta+2\mu_r+\delta_r+\gamma_h+(1-\theta)\delta_h+(1-p)\cdot (\eta_1\bar{k}+\eta_2)\cr\cr&&+\eta_3+(1-p)^2\cdot\frac{\sigma^2_1}{2}\cdot{\bar{k}}^2+(1-p)^2\sigma^2_2+\frac{\sigma^2_3}{2}+\frac{\sigma^2_4}{2}+\frac{\sigma^2_5}{2}+\frac{\sigma^2_6}{2}+\frac{\sigma^2_7}{2}+\frac{\sigma^2_8}{2}.
\end{eqnarray*}

\noindent Integrating both sides of equation (\ref{leeq3.1})  from 0 to
$\tau_m \wedge t=\min\{\tau_m,t\}$, where $t\in[0,T]$, and by  (\ref{leq3.2}) we obtain the following relation
\begin{eqnarray*}
\int^{\tau_m \wedge
t}_0d\mathcal{V}\bigg(\mathcal{X}(r)\bigg) &\leq&
\int^{\tau_m \wedge
t}_0 Kdr
-(1-p)\int^{\tau_m \wedge
t}_0\frac{(S_h(r)-1)\sigma_1I_r(r)}{N_h(r)}dB_1(r)\cr\cr&&+(1-p)\int^{\tau_m \wedge
t}_0\frac{(I_h(r)-S_h(r))\sigma_2}{N_h(r)}dB_2(r)\cr\cr&&+
\int^{\tau_m \wedge
t}_0
(S_h(r)-1)\sigma_3dB_3(r)+\int^{\tau_m \wedge
t}_0(I_h(r)-1)\sigma_4dB_4(r)\cr\cr&&+\int^{\tau_m \wedge
t}_0
(Q_h(r)-1)\sigma_5dB_5(r)+\int^{\tau_m \wedge
t}_0
(R_h(r)-1)\sigma_6dB_6(r)\cr\cr&&+\int^{\tau_m \wedge
t}_0 (S_r(r)-1)\sigma_7dB_7(r)+\int^{\tau_m \wedge
t}_0 (I_r(r)-1)\sigma_8dB_8(r),
\end{eqnarray*}
since $\mathcal{X}(\tau_m \wedge
t)\in \mathbb{R}^6_+$, $t\in[0,T]$. 

The rest of the proof follows the arguments in Mao et. al. \cite{Mao2}. Taking the expectation on both sides above, we have 
\begin{eqnarray}\label{rel1}
\mathrm{E}\mathcal{V}\left(\mathcal{X}(\tau_m\wedge T)\right)&\leq& \mathcal{V}(\mathcal{X}(0))+ \mathrm{E}\int^{\tau_m \wedge
T}_0 Kdr\cr\cr&&=\mathcal{V}(\mathcal{X}(0))+ K \mathrm{E} (\tau_m\wedge T) \leq
\mathcal{V}(\mathcal{X}(0))+KT.
\end{eqnarray}

\noindent Set $A_m=\{\tau_m \leq t\}$ for $m \geq m_1 $. By  (\ref{eq3.1}),
we have $P(A_m) \geq \varepsilon$.
Note that for every $\omega \in A_m$, there is at least one
term among these five $S_h(\tau_m\wedge
T)$, $I_h(\tau_m\wedge T)$, $Q_h(\tau_m\wedge
T)$, $R_h(\tau_m\wedge T)$, $S_r(\tau_m\wedge T)$ and $I_r(\tau_m\wedge T)$ which is
equal to either $m$ or $\dfrac{1}{m}$. Then
$$\mathcal{V}(\mathcal{X}(\tau_m\wedge T)) \geq
\left(\left(m-1-\log m\right)\wedge\left(\frac{1}{m}-1-\log \frac{1}{m}\right)\right).
$$
Therefore, from (\ref{eq3.1}) and (\ref{rel1}),  we can easily obtain
\begin{eqnarray*}
\mathcal{V}(\mathcal{X}(0))+KT &\geq&  \mathrm{E}\left(\mathbb{I}_{A_m(\omega)}
\mathcal{V}(\mathcal{X}(\tau_m))\right)
\cr\cr&&\geq\varepsilon\left(\left(m-1-\log m\right)\wedge\left(\frac{1}{m}-1-\log \frac{1}{m}\right)\right),
\end{eqnarray*}
where $\mathbb{I}_{A_m(\omega)}$  is the indicator function of $A_m$.  Letting $m$ increase to infinity, we find a   contradiction: 
 $\infty\leq\mathcal{V}(\mathcal{X}(0))+KT=\infty.$ So it must be $\tau_{\infty}=\infty$ a.s.
 The non-negativity of the solution is guaranteed as a product of the argument above thanks to the definition of $(\tau_m)_{m\geq m_0}$ and $\tau_{\infty}=\infty$ a.s.
 \end{proof}

\noindent The result above shows that the proposed model is well-posed and provides plausible future scenarios for the evolution of the disease depending on the initial
circumstances (i.e., initial conditions), representing the population conditions at the start of the epidemic, and the set of parameters described in Table 1.

Particularly, it is important that the trajectories keep positive since a negative solution would not have meaningful interpretation.

\vspace{0.5cm}

\section{Stochastically Ultimately Bounded and Permanent}
\noindent After establishing the well-posedness  of the system (\ref{STD1.1}), we proceed detailing other properties.   
 Now, we present the definition of stochastic ultimate boundedness and permanent solution  of the system, see \cite{Cai}.
\begin{definition}\label{SUB}
     The solutions $\mathcal{X}(t)$, $t\geq0$, of model (\ref{STD1.1}) are said to be \textbf{stochastically ultimately bounded}, if for any $\epsilon\in(0,1)$, there is a positive constant $\kappa = \kappa(\epsilon)$, such that for any
initial value $\mathcal{X}(0) \in \mathbb{R}^6_+ $, the solution to (\ref{STD1.1}) satisfies
$$\underset{t\to \infty} {\lim \sup} \,\mathrm{P}\left\{|\mathcal{X}(t)|>\kappa\right\}<\epsilon.$$
\end{definition}
\begin{definition}\label{per}
The solutions $\mathcal{X}(t)$, $t\geq0$ of model (\ref{STD1.1})
are said to be \textbf{stochastically permanent} if for any $\epsilon \in (0, 1)$,
there exists a pair of positive constants $\kappa = \kappa(\epsilon)$ and $\chi = \chi(\epsilon)$,
such that for any initial value $\mathcal{X}(0) \in \mathbb{R}^6_+ $, the solution
to (\ref{STD1.1}) has the following properties
$$\underset{t\to \infty} {\lim \inf} \,\mathrm{P}\left\{|\mathcal{X}(t)|\geq\kappa\right\}\geq 1-\epsilon;$$
$$\underset{t\to \infty} {\lim \inf} \,\mathrm{P}\left\{|\mathcal{X}(t)|\leq\chi\right\}\geq 1-\epsilon.$$
\end{definition}

\begin{proposition}
For any initial value $\mathcal{X}(0) \in \R^6_+$, let model (\ref{STD1.1}) satisfy 
\begin{equation}\label{K}
\mathcal{K}(t)=N_h(t)+N_r(t)>0\,\,\, \text{for all }\,\,\, t\geq0.
\end{equation}
Then the solution $\mathcal{X}(t)$, $t\geq 0$ is stochastically ultimately bounded and permanent.
\end{proposition}
\noindent From a biological point of view, assumption (\ref{K}) is the relevant case, since $\underset{t\geq0}\inf\, \mathcal{K}(t)=0$ would also mean that  the populations  are extinct.
\begin{proof}
Define $\mathcal{V}(\mathcal{X}(t))=\mathcal{K}(t)+\frac{1}{\mathcal{K}(t)}$. Applying the It\^o formula,  we get
\begin{eqnarray*}\label{Int}
&&d\mathcal{V}(\mathcal{X}(t))=L\mathcal{V}(\mathcal{X}(t))dt\cr\cr&&+\left(1-\frac{1}{\mathcal{K}^2(t)}\right)\bigg[ \frac{(p-1)\sigma_1I_r(t)S_h(t)}{N_h(t)}dB_1(t)+\sigma_3S_h(t)dB_3(t)+\sigma_4I_h(t)dB_4(t)\cr\cr&&+\sigma_5Q_h(t)dB_5(t)+\sigma_6R_h(t)dB_6(t)+\sigma_7S_r(t)dB_7(t)
+\sigma_8I_r(t)dB_8(t)\bigg],
\end{eqnarray*}
where
\begin{eqnarray*}
&&L\mathcal{V}(\mathcal{X}(t))=\left(1-\frac{1}{\mathcal{K}^2(t)}\right)\bigg(\theta_h-\mu_h(S_h(t)+I_h(t)+Q_h(t)+R_h(t))\cr\cr&&-\delta_hI_h(t)-(1-\theta)\delta_hQ_h(t)+\theta_r-\mu_r(S_r(t)+I_r(t))-\delta_rI_r(t)\bigg) \cr\cr&&+\frac{\sigma^2_1(1-p)^2I^2_r(t)S^2_h(t)}{N^2_h\mathcal{K}^3(t)}+\frac{2\sigma^2_2(1-p)^2I^2_h(t)S^2_h(t)}{N^2_h\mathcal{K}^3(t)}\cr\cr&&+
\frac{\sigma^2_3S^2_h(t)}{\mathcal{K}^3(t)}+\frac{\sigma^2_4I^2_h(t)}{\mathcal{K}^3(t)}+\frac{\sigma^2_5Q^2_h(t)}{\mathcal{K}^3(t)}+\frac{\sigma^2_6R^2_h(t)}{\mathcal{K}^3(t)}+\frac{\sigma^2_7S^2_r(t)}{\mathcal{K}^3(t)}+\frac{\sigma^2_8I^2_r(t)}{\mathcal{K}^3(t)}.
\end{eqnarray*}

\noindent We can rewrite the function $ L\mathcal{V}(\mathcal{X}(t)),\, t\geq 0,$ as 
\begin{eqnarray*}
&&L\mathcal{V}(\mathcal{X}(t))
=\theta_h-\delta_hI_h(t)+\theta_r-(1-\theta)\delta_hQ_h(t)+\frac{(1-\theta)\delta_h Q_h(t)}{\mathcal{K}^2(t)}\cr\cr&&-H_1\left(\mathcal{K}(t)+\frac{1}{\mathcal{K}(t)}\right)+\frac{2H_1}{\mathcal{K}(t)}+\frac{\delta_r I_r(t)}{\mathcal{K}^2(t)}\cr\cr&&+
\frac{\sigma^2_1(1-p)^2I^2_r(t)S^2_h(t)}{N^2_h\mathcal{K}^3(t)}+\frac{2\sigma^2_2(1-p)^2I^2_h(t)S^2_h(t)}{N^2_h\mathcal{K}^3(t)}\cr\cr&&+
\frac{\sigma^2_3S^2_h(t)}{\mathcal{K}^3(t)}+\frac{\sigma^2_4I^2_h(t)}{\mathcal{K}^3(t)}+\frac{\sigma^2_5Q^2_h(t)}{\mathcal{K}^3(t)}+\frac{\sigma^2_6R^2_h(t)}{\mathcal{K}^3(t)}+\frac{\sigma^2_7S^2_r(t)}{\mathcal{K}^3(t)}+\frac{\sigma^2_8I^2_r(t)}{\mathcal{K}^3(t)},
\end{eqnarray*}
for any $t$, with $H_1=\max\{\mu_h,\mu_r\}$. Since (\ref{K}) holds, and $\dfrac{S_h}{N_h}<1$, we find the upper estimate:
\begin{eqnarray*}
 L\mathcal{V}(\mathcal{X}(t))\leq \mathcal{F}-H_1\left(\mathcal{K}(t)+\frac{1}{\mathcal{K}(t)}\right),
\end{eqnarray*}
which is
\begin{equation}\label{t1}
    L\mathcal{V}(\mathcal{X}(t))\leq \mathcal{F}-H_1\mathcal{V}(t),
\end{equation}
where $\mathcal{F}=\theta_h+\theta_r+\dfrac{\delta_h+\delta_r+2H_1}{\underset{t\geq0}\inf\,\mathcal{K}(t)}$.

\noindent Let  $\mathcal{V}_1(\mathcal{X}(t))=e^{H_1t}\mathcal{V}(\mathcal{X}(t))$, using It\^o formula,  by (\ref{t1}) we have
 \begin{eqnarray*}
d(e^{H_1t}\mathcal{V}(\mathcal{X}(t)) &\leq&  H_1e^{H_1t}\mathcal{V}(\mathcal{X}(t)))dt+e^{H_1t}(\mathcal{F}-H_1\mathcal{V}(\mathcal{X}(t)))dt \cr\cr&& +e^{H_1t}\left(1-\frac{1}{\mathcal{K}^2(t)}\right)\bigg[ \frac{(p-1)\sigma_1I_r(t)S_h(t)}{N_h(t)}dB_1(t)\cr\cr&&+\sigma_3S_h(t)dB_3(t)+\sigma_4I_h(t)dB_4(t)+\sigma_5Q_h(t)dB_5(t)\cr\cr&&+\sigma_6R_h(t)dB_6(t)+\sigma_7S_r(t)dB_7(t)
+\sigma_8I_r(t)dB_8(t)\bigg] \cr\cr
 &&=e^{H_1t}\mathcal{F}dt+e^{H_1t}\left(1-\frac{1}{\mathcal{K}^2(t)}\right)\bigg[\frac{(p-1)\sigma_1I_r(t)S_h(t)}{N_h(t)}dB_1(t)\cr\cr
 &&+\sigma_3S_h(t)dB_3(t)+\sigma_4I_h(t)dB_4(t)+\sigma_5Q_h(t)dB_5(t)\cr\cr
 &&+\sigma_6R_h(t)dB_6(t)+\sigma_7S_r(t)dB_7(t)
+\sigma_8I_r(t)dB_8(t)\bigg].
\end{eqnarray*}
\noindent Then
\begin{eqnarray*}
\mathrm{E}(\mathcal{V}(\mathcal{X}(t)))&=&\mathrm{E}\left(e^{-H_1t}\mathcal{V}(\mathcal{X}(0))+ \frac{\mathcal{F}}{H_1}\right)\leq\frac{\mathcal{F}}{H_1}+\mathcal{V}(\mathcal{X}(0))=\alpha.
\end{eqnarray*}
\noindent Let's fix $\epsilon>0$, then exist $\kappa>0$ such that $\dfrac{\alpha}{\kappa}<\epsilon$. By Markov's inequality,  for a sufficiently large constant $\kappa$  we have
\begin{equation*}
  P\left\{\mathcal{K}(t)+\frac{1}{\mathcal{K}(t)} > \kappa\right\} \leq \frac{1}{\kappa}\mathrm{E}\left(\mathcal{K}(t)+\frac{1}{\mathcal{K}(t)}\right) \leq  \frac{\alpha}{\kappa}<\epsilon. 
\end{equation*}

\noindent Then
\begin{equation*}
\lim_{t \to \infty}\sup \mathrm{P}\left\{\mathcal{K}(t)+\frac{1}{\mathcal{K}(t)} > \kappa\right\} \leq \epsilon.
\end{equation*}

\noindent That is
\begin{equation}\label{mi}
\lim_{t \to \infty} \sup \mathrm{P}\{\mathcal{K}(t) > \kappa\} < \epsilon.
\end{equation}
Observe that, for $x\in \mathbb{R}^d_+$, we have $ |x|\leq\sqrt{d}\sum_{i=1}^dx_i$. Then by (\ref{mi}), we get
\begin{equation*}
\lim_{t \to \infty} \sup \mathrm{P}\{|\mathcal{X}(t)| > \kappa\sqrt{6}\} \leq
\lim_{t \to \infty} \sup \mathrm{P}\{\mathcal{K}(t) > \kappa\}\leq\epsilon.
\end{equation*}
According to the Definition \ref{SUB},  it means that  the solutions $\mathcal{X}(t)$,  $t\geq0$, of model (\ref{STD1.1}) are stochastically ultimately bounded. Similarly  can be easily obtained the stochastically
permanent property for solutions according to the Definitions  \ref{per}  (see \cite{Cai}). This completes the proof.
\end{proof}

\noindent  Stochastically ultimately boundedness indicates that the model's depiction of population dynamics won't show unbounded expansion or decline in the actual situation of monkeypox disease transmission. The system's variables, which represent various population segments, have certain upper bounds that they will not ultimately surpass. According to the concept of stochastic eventually boundedness, there is a strong likelihood that the solutions of the stochastic differential equations driving the monkeypox model will remain within certain bounds throughout time.
\section{Study of Extinction of the Disease}
\noindent In this section, we determine the factors that could lead to the extinction of the disease.
\begin{theorem}\label{Exti}
Let $\mathcal{X}(t)$, $t\geq0$, be a solution of
system (\ref{STD1.1}) with initial value $\mathcal{X}(0)\in\mathbb{R}^6_+$. Assume that $$\mathcal{R}_0=\frac{(1-p)(\eta_1+\eta_2)+\eta_3}{\min\{\mu_h,\mu_r\}+\min\{\delta_h,\delta_r\}} <1,$$
  then the number of infected individuals of system (\ref{STD1.1}) goes to extinction almost surely, i.e.,
 \begin{equation}\label{sum} \underset{t\to\infty}{\lim }(I_h(t)+Q_h(t)+I_r(t))=0\,\,\,\textnormal{a.s.}
\end{equation}

\end{theorem}
\begin{proof} 
The It\^o's formula applied to $\mathcal{U}=ln(I_h+Q_h+I_r)$,
where $\mathcal{U}:\R^3_+\to \R_+$,   leads to
\begin{eqnarray}\label{du_hole}
d\mathcal{U}(t)&=&L\mathcal{U}(t)dt+
\frac{(1-p)\sigma_2I_h(t)S_h(t)dB_2(t)}{N_h(t)(I_h(t)+Q_h(t)+I_r(t))}\cr\cr&&+\frac{\sigma_4I_h(t)dB_4(t)+\sigma_5Q_h(t)dB_5(t)+\sigma_8I_r(t)dB_8(t)}{I_h(t)+Q_h(t)+I_r(t)}, 
\end{eqnarray}
where
\begin{eqnarray}\label{luest}
 && L\mathcal{U}(t)= \frac{(1-p)}{N_h(t)(I_h(t)+Q_h(t)+I_r(t))}\cdot\Big(\eta_1I_r(t)S_h(t)+\eta_2I_h(t)S_h(t)\Big)\cr\cr&&+\frac{\eta_3S_r(t)I_r(t)}{N_r(t)(I_h(t)+Q_h(t)+I_r(t))}\cr\cr&&-\frac{(\mu_h+\delta_h)I_h(t)
+(\mu_h+\gamma_h+(1-\theta)\delta_h)Q_h(t)+(\mu_r+\delta_r)I_r(t)}{I_h(t)+Q_h(t)+I_r(t)}\cr\cr&&-\frac{(1-p)^2S^2_h(t)\sigma^2_2I^2_h(t)}{2N^2_h(t)(I_h(t)+Q_h(t)+I_r(t))^2}
-\frac{\sigma^2_4I^2_h(t)+\sigma^2_5Q^2_h(t)+\sigma^2_8I^2_r(t)}{2(I_h(t)+Q_h(t)+I_r(t))^2}
.
\end{eqnarray}
Since $S_h\leq N_h$, $Q_h\leq N_h$, from (\ref{luest}), we get
\begin{eqnarray*}
 L\mathcal{U}(t)&\leq&(1-p)(\eta_1+\eta_2)+\eta_3-(\min\{\mu_h,\mu_r\}+\min\{\delta_h,\delta_r\} )
\end{eqnarray*}
Integrating on both side equation (\ref{du_hole}), and substituting (\ref{luest}), we obtain
\begin{eqnarray*}
 && \mathcal{U}(t)- \mathcal{U}(0)\leq (\min\{\mu_h,\mu_r\}+\min\{\delta_h,\delta_r\} )(\mathcal{R}_0-1)\, t\cr\cr&&+ (1-p)\sigma_2B_2(t)+\sigma_4B_4(t)+\sigma_5B_5(t)+\sigma_8B_8(t).
 \end{eqnarray*}
By the strong low of large number (see \cite{GS}), for $ i=2,4,5,8,$  we have
\begin{equation}\label{win}
    \underset{t\to\infty}{\lim}\frac{B_i(t)}{t}=0 \,\,\,\textnormal{a.s.}
\end{equation}
Since $\mathcal{R}_0<1$ (by assumption),   $L(\mathcal{U})(t) \leq 0$, and by (\ref{win}) we obtain
\begin{equation*}
    \underset{t\to\infty}{\lim \sup}\frac{\mathcal{U}(t)}{t}\leq(\min\{\mu_h,\mu_r\}+\min\{\delta_h,\delta_r\} )(\mathcal{R}_0-1)<0 \,\,\,\textnormal{a.s.}
\end{equation*}
leading to (\ref{sum}), which completes the proof.
\end{proof}
%%%%%%%%%%%%%%%%%%%%%%%%%
To conform the analytical results above, we  provide some numerical illustrations. First we demonstrate the case for $R_0<1$.
\begin{example} \label{4.2}
Consider the model (\ref{STD1.1}) with parameters value:  
\begin{itemize}
    \item $(S_h(0), Q_h(0), I_h(0), R_h(0), S_r(0), I_r(0))=(90,50,60,70,80,30);$
    \item $(\sigma_1,~ \sigma_1,~ \sigma_3,~ \sigma_4, ~\sigma_5, ~\sigma_6,~ \sigma_7,~ \sigma_8)=(0.05, 0.04, 0.01, 0.05, 0.04, 0.01, 0.05, 0.04).$
\end{itemize}
We also consider the values from  Table \ref{Table 2} that will result in producing, $R_0<1$, where the conditions of Theorem \ref{Exti} are satisfied.

\begin{table}[H]
   \centering
  \begin{tabular}{|c|c|}
    \hline
    Model's Elements & Particular Values of Elements \\
    \hline
$\theta$ & 0.043 \\
$\theta_h$   & 10\\
$p$   & 0.041 \\
$\eta_1$   & 0.009\\
$\eta_2$   & 0.002 \\
$\eta_3$   & 0.0027\\
$\mu_h$   & 0.05\\
$\delta_h$ & 0.003\\
$\zeta$ & 0.5 \\
$\gamma_h$ & 0.2\\
$\theta_r$ & 10\\
$\mu_r$ & 0.02\\
$\delta_r$ & 0.004\\
    \hline
  \end{tabular}
  \caption{Value of the elements}\label{Table 2}
\end{table}
\noindent The graphs are obtained while using python software for $$R_0 =\frac{(1-0.041)(0.009+0.002)+0.0027}{0.02+0.003}= 0.567<1.$$ 
\end{example}
\noindent The obtained curves show the dynamics of the monkeypox disease transmission between humans and rodents population by comparing the deterministic and stochastic versions of the model.  Also the uncertainty is quantified through the means  and standard deviation of $S_h$, $I_h$, $Q_h$, $R_h$, $S_r$ and $I_r$. Standard deviation curves represent the variability of each compartment over multiple simulations, indicating the uncertainty associated with each curve. 

%%%%%%%%%%%%%%%%%
\begin{figure}[H]\label{fig20}
\center{\includegraphics[scale=0.60]{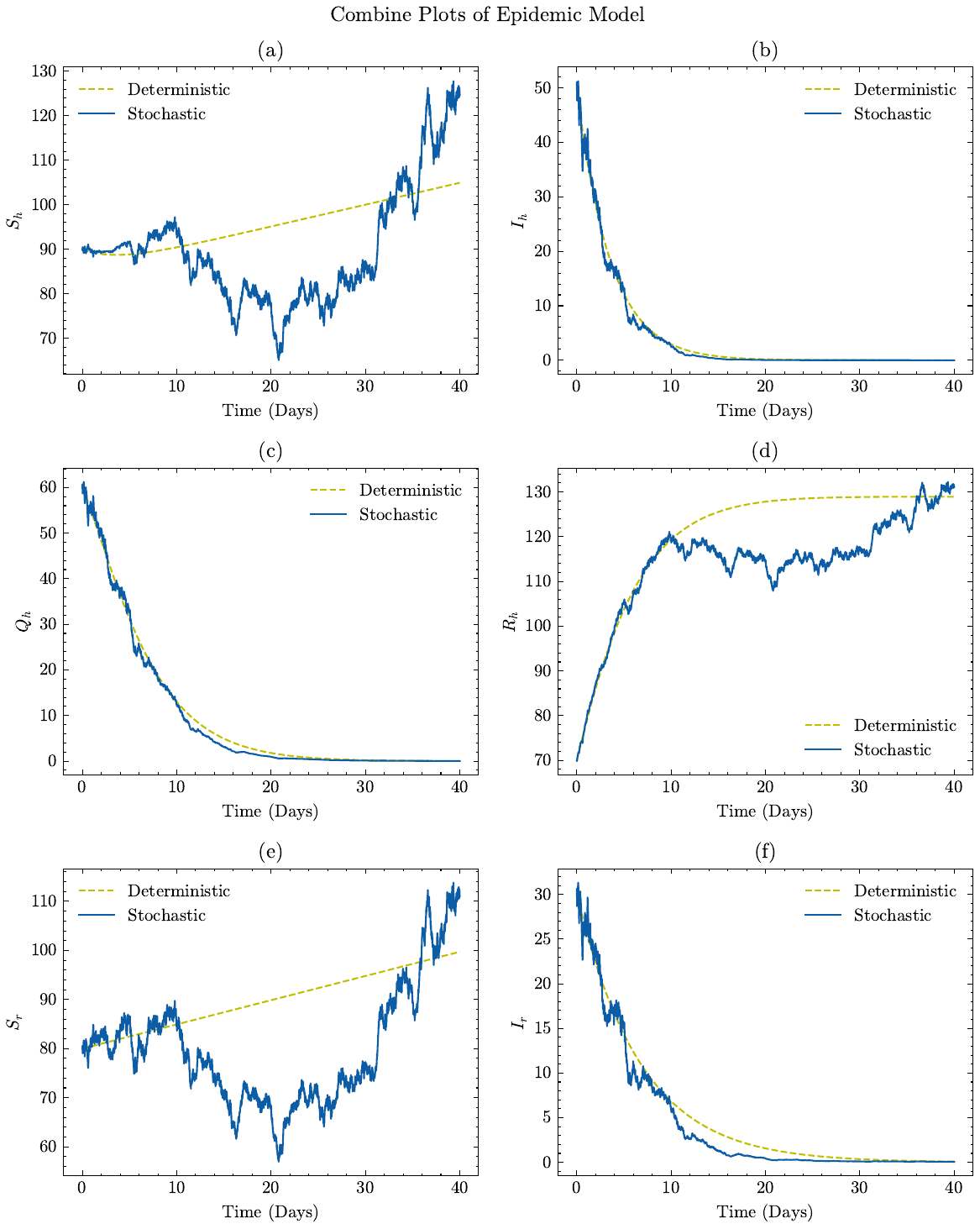}}
\caption{The curves represent the behaviour of susceptible(a), infected(b), isolated(c), recovered(d) humans populations and of different rodents classes (e)-(f) as time evolve.}
\end{figure}

\begin{figure}[H]\label{fig41}
\center{\includegraphics[scale=0.60]{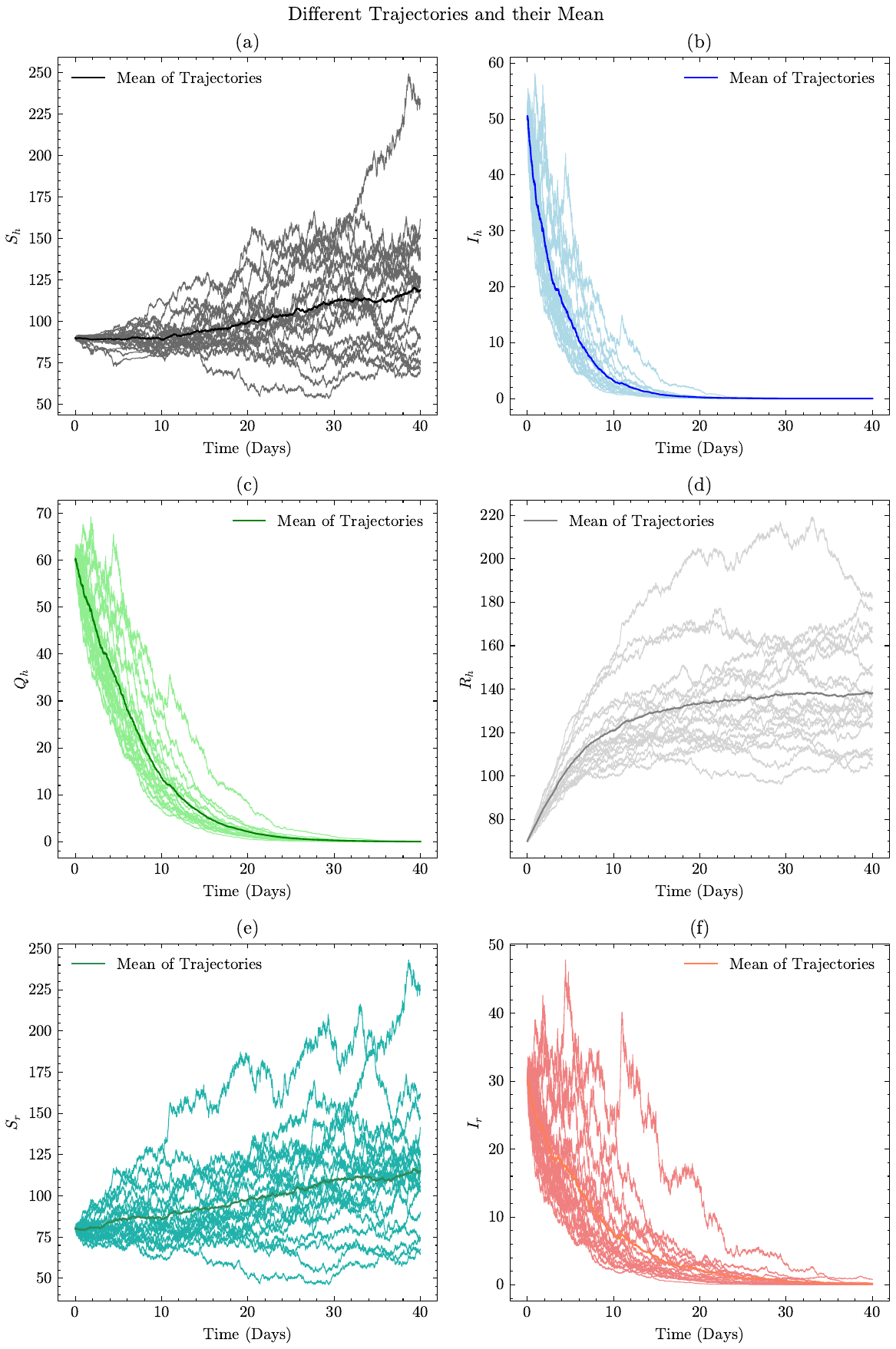}}
\caption{Paths of the different classes of humans (a)-(d) and rodents (e)-(f) along with mean of each family of curves. Means are shown with dark curves surrounded by family of related trajectories }
\end{figure}
\begin{figure}[H]\label{fig51}
\center{\includegraphics[scale=0.60]{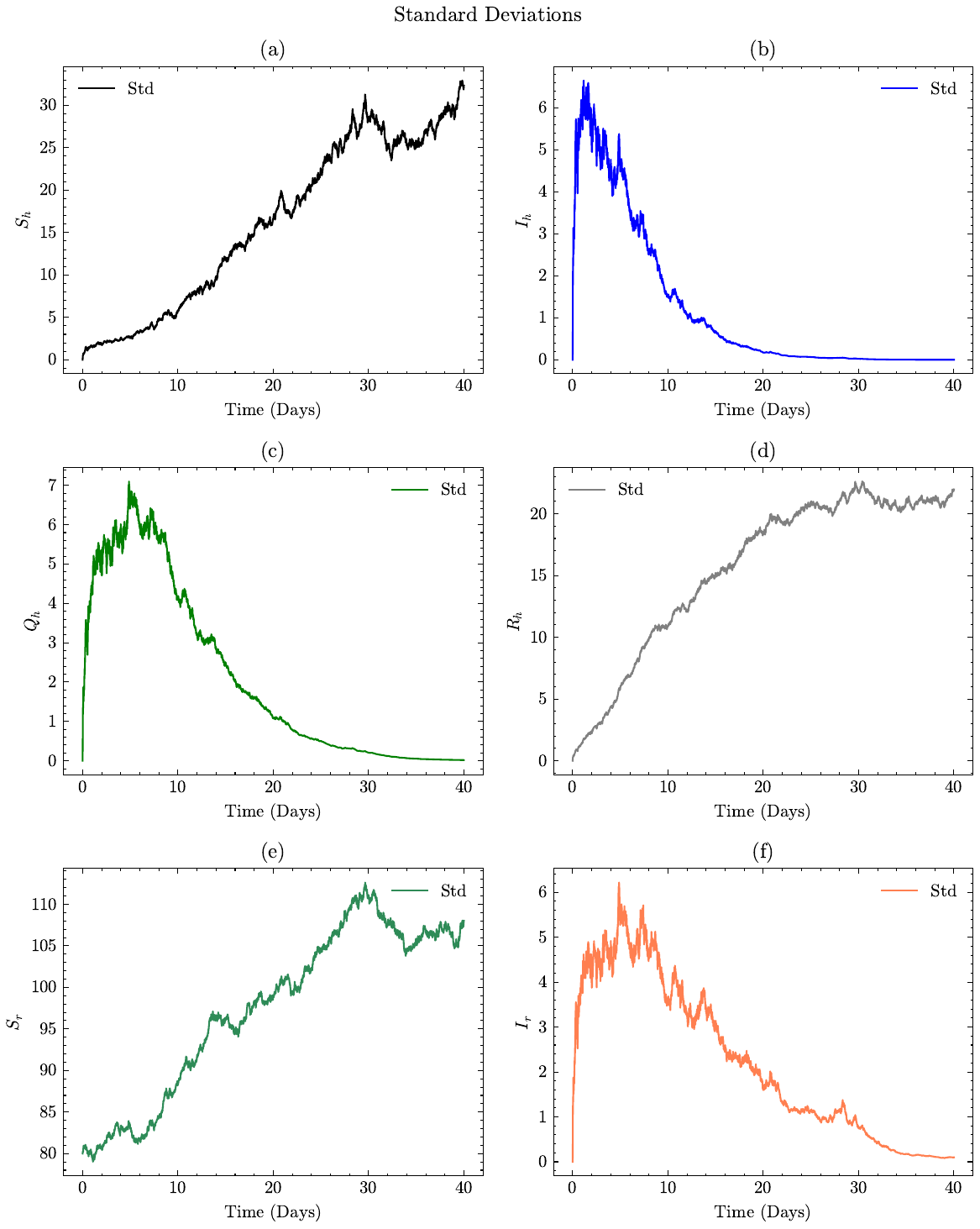}}
\caption{Curves show the variability of each compartment  humans (a)-(d) and rodents (e)-(f) over multiple simulations, indicating the uncertainty associated with each curve.}
\end{figure}
\begin{figure}[H]\label{fig6}
\center{\includegraphics[scale=0.50]{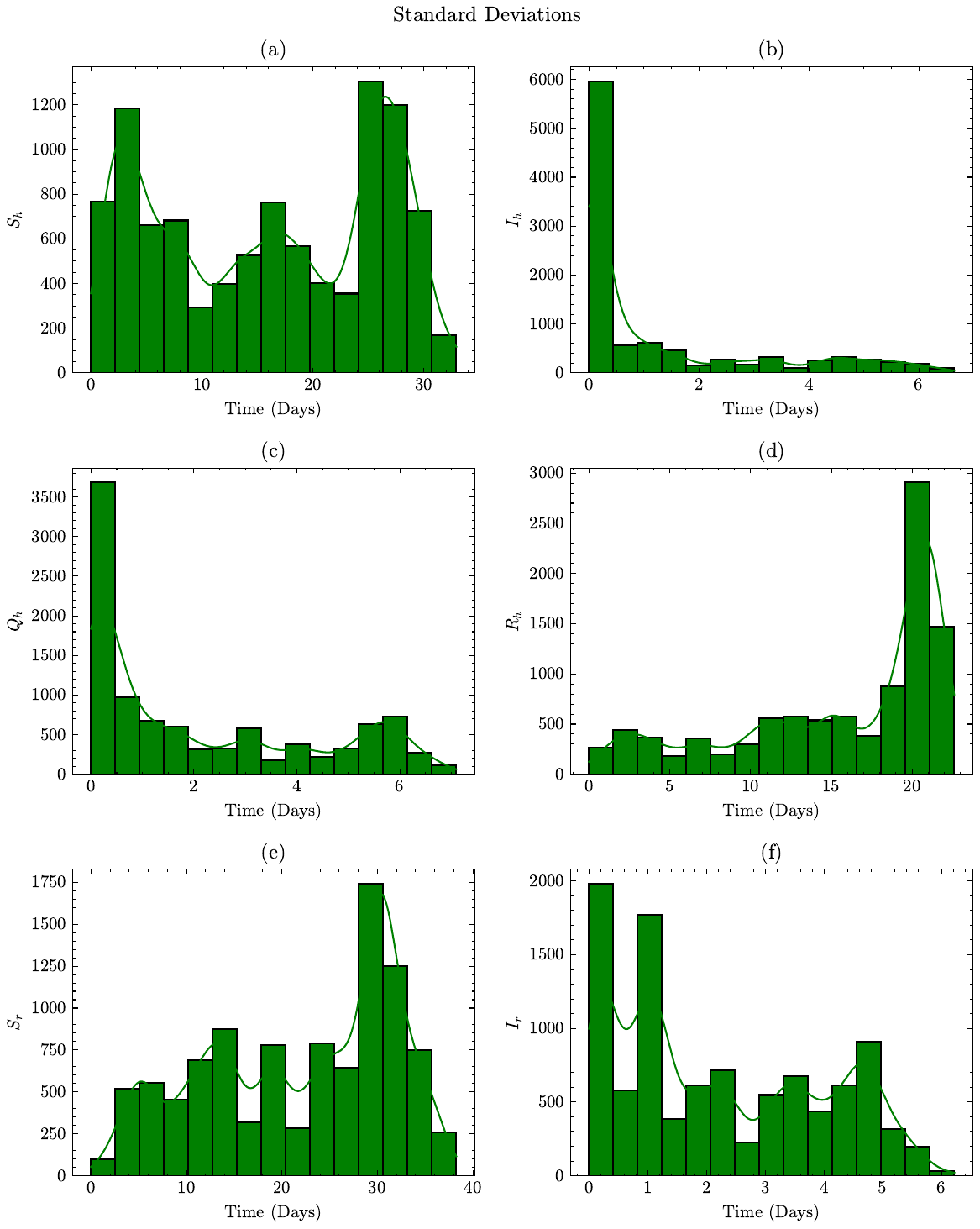}}
\caption{The graphs show the distribution of values for each compartment  humans (a)-(d) and rodents (e)-(f): The histograms for Susceptible Humans(a) and Susceptible Rodents (e) indicate a high number of individuals in these compartments, suggesting a large pool of potential hosts for the disease. The histograms for Infected Humans (b) and Infected Rodents (f) show a smaller number of individuals, indicating the disease prevalence. The histograms for Isolated Humans (c) and Recovered Humans (d) indicate the effectiveness of isolation and recovery measures.}
\end{figure}
\noindent Next we discuss the case for $R_0>1$.
\begin{example}
 To illustrate this we consider the formulated model (\ref{STD1.1}) with the same initial values and volatility $\sigma_i, i=1,...,8$ defined in Example \ref{4.2}.  We also consider the  values from Table \ref{Table 3} that will results in producing, $$R_0 =\frac{(1-0.041)(0.009+0.002)+0.027}{0.02+0.003}= 1.6326>1.$$
 \noindent In particular, we remark that in Example \ref{4.2} the value $\eta_3=0.0027$, while here it is $\eta_3=0.027.$
\begin{table}[H]
   \centering
  \begin{tabular}{|c|c|}
    \hline
    Model's Elements & Particular Values of Elements \\
    \hline
$\theta$ & 0.043 \\
$\theta_h$   & 10\\
$p$   & 0.041 \\
$\eta_1$   & 0.009\\
$\eta_2$   & 0.002 \\
$\eta_3$   & 0.027\\
$\mu_h$   & 0.05\\
$\delta_h$ & 0.003\\
$\zeta$ & 0.5 \\
$\gamma_h$ & 0.2\\
$\theta_r$ & 10\\
$\mu_r$ & 0.02\\
$\delta_r$ & 0.004\\
    \hline
  \end{tabular}
  \caption{Descriptions values of the elements}\label{Table 3}
\end{table}
\noindent All trajectories after long time (800 days) with $R_0>1$ have different nature as compared to the case when, $R_0<1$. Indeed, for $R_0<1$, it is shown in Example \ref{4.2} that all trajectories reach the endemic equilibrium point after 40 days. 
\end{example}

\begin{figure}[H]\label{fig20}
\center{\includegraphics[scale=0.60]{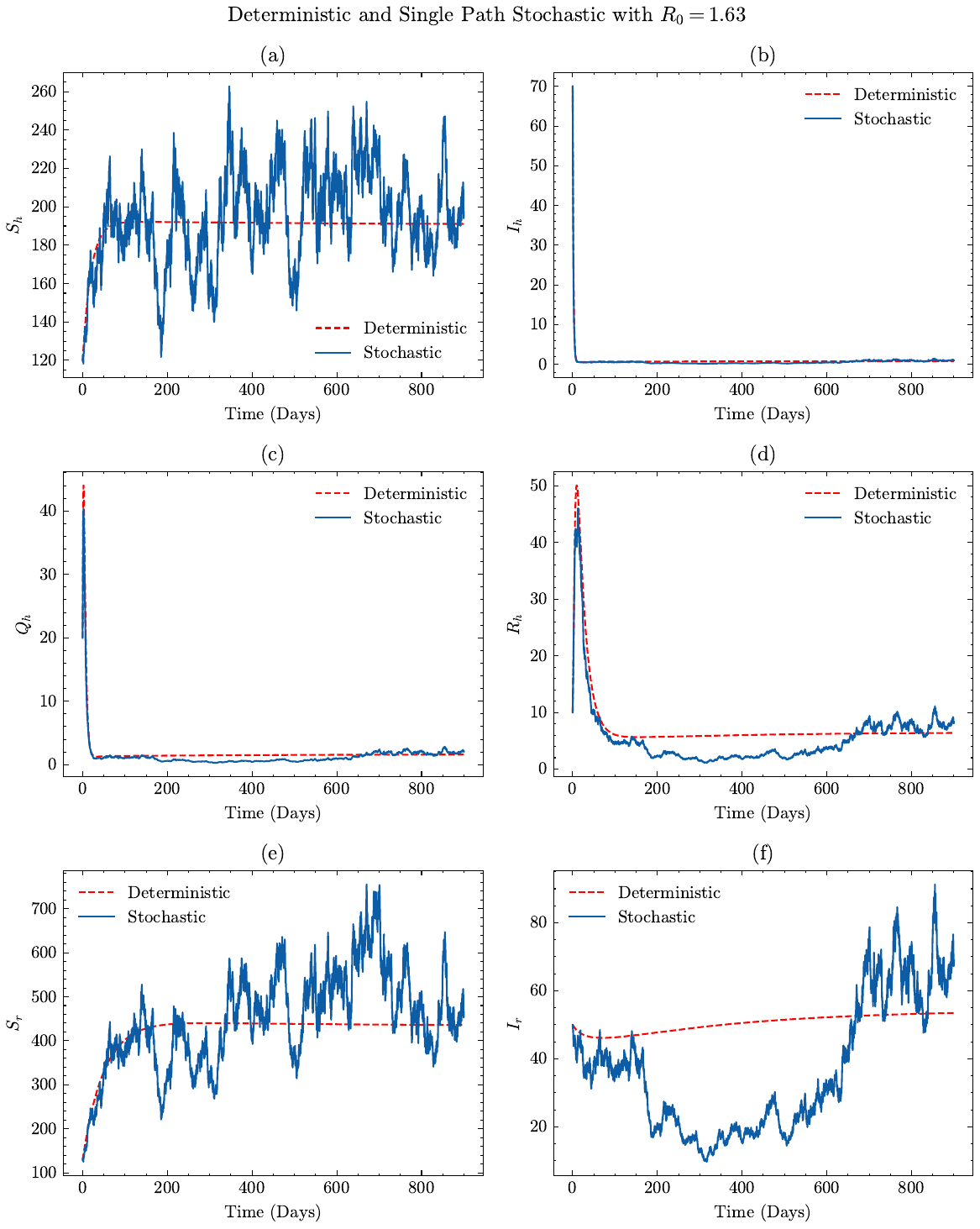}}
\caption{For $R_0>1$: the curves represent the behavior of susceptible(a), infected(b), isolated(c), recovered(d) humans populations and of different rodents classes (e)-(f) as time evolve for 800 days.}
\end{figure}
\begin{figure}[H]\label{fig3}
\center{\includegraphics[scale=0.60]{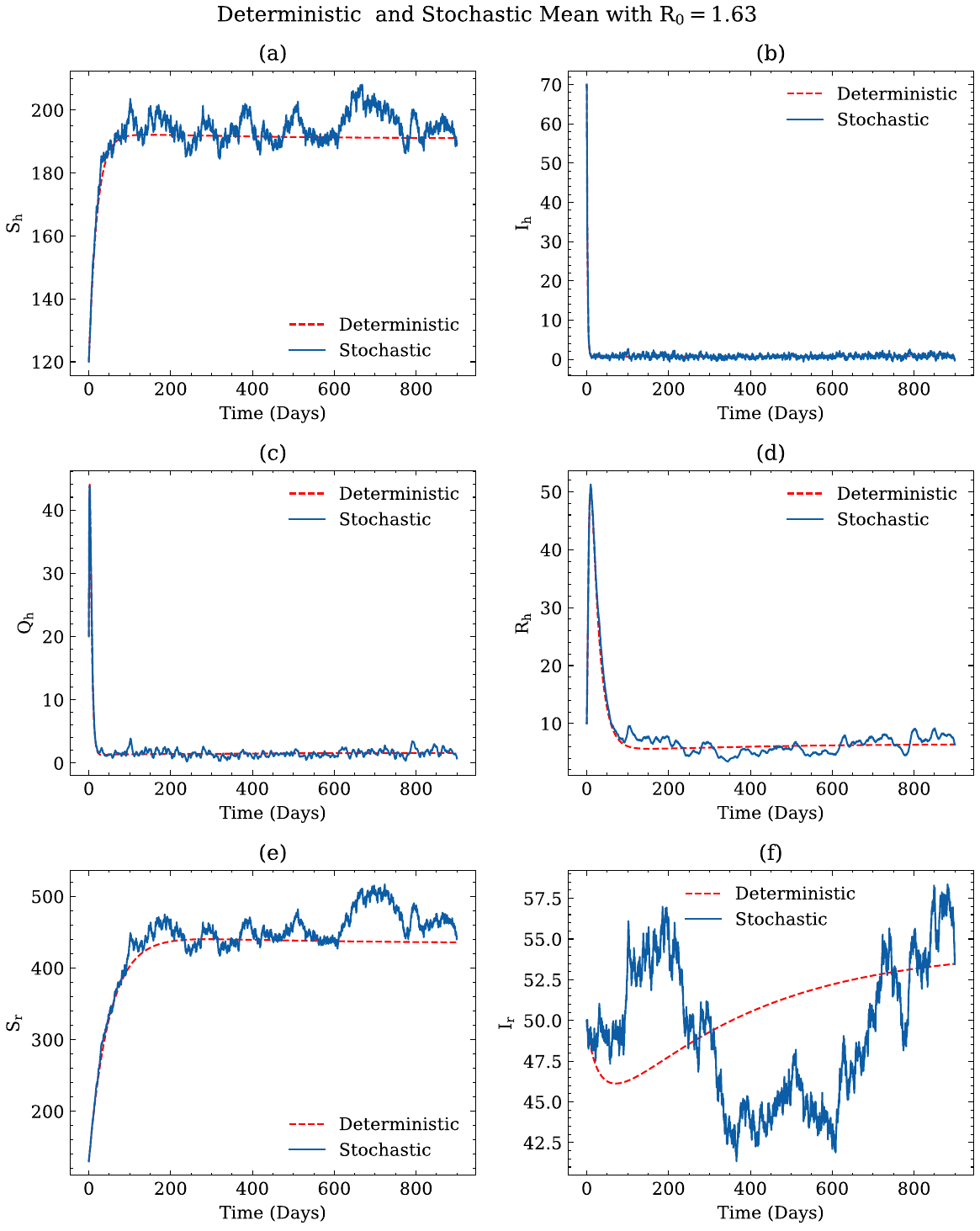}}
\caption{ For $R_0>1$: the curves without noise and mean of cures with noise for human population of susceptible(a), infected(b), isolated(c), recovered(d) humans populations and of different rodents classes (e)-(f) as time evolve for 800 days. }
\end{figure}
\begin{figure}[H]\label{fig4}
\center{\includegraphics[scale=0.60]{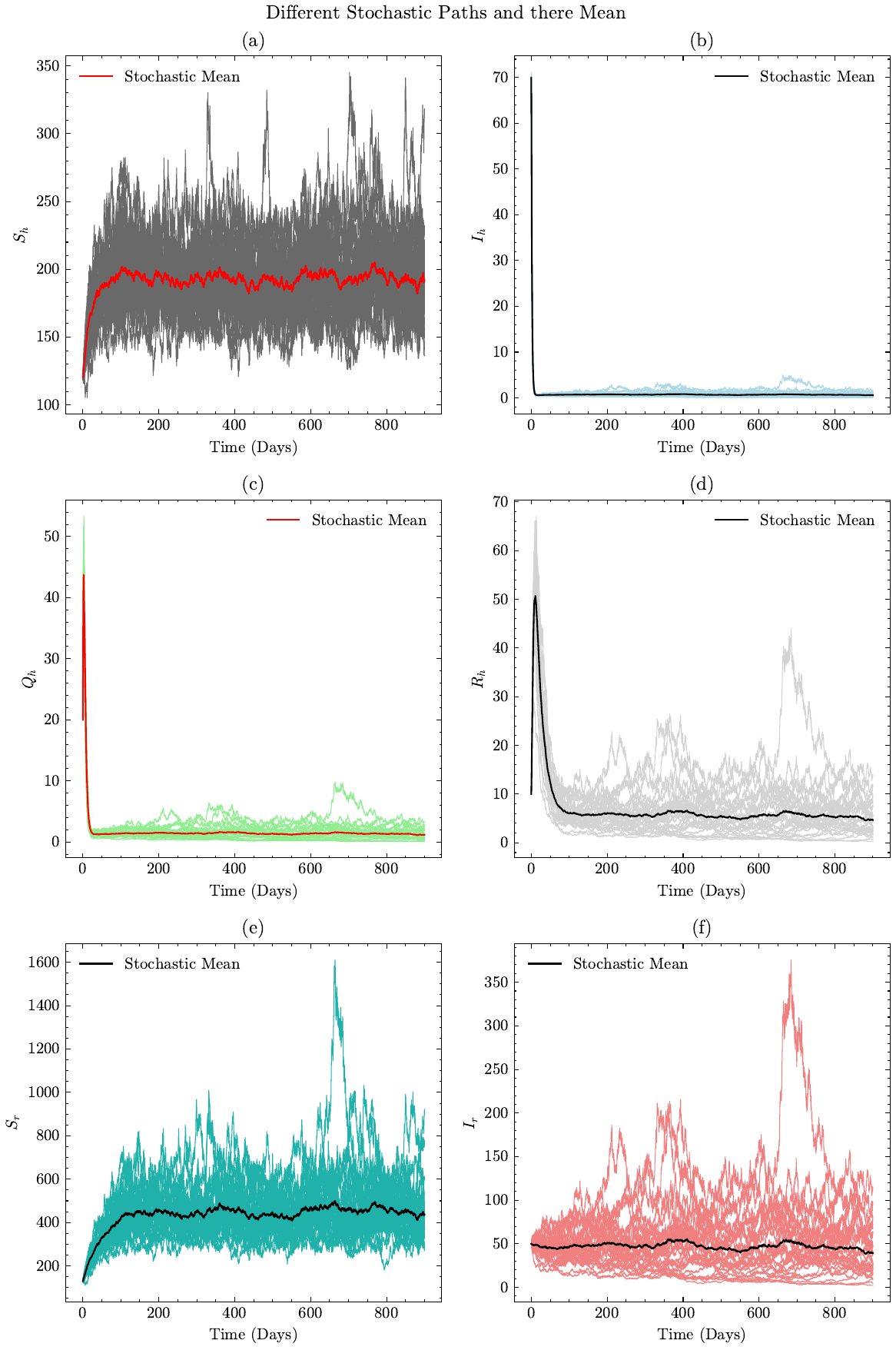}}
\caption{ For $R_0>1$, paths of the different classes of humans (a)-(d) and rodents (e)-(f) along with mean of each family of curves. Means of each family of paths are shown with dark curves surrounded by family of related trajectories. After a long time 800 days the infected population still fluctuate.}
\end{figure}
\newpage
\begin{figure}[H]\label{fig5}
\center{\includegraphics[scale=0.60]{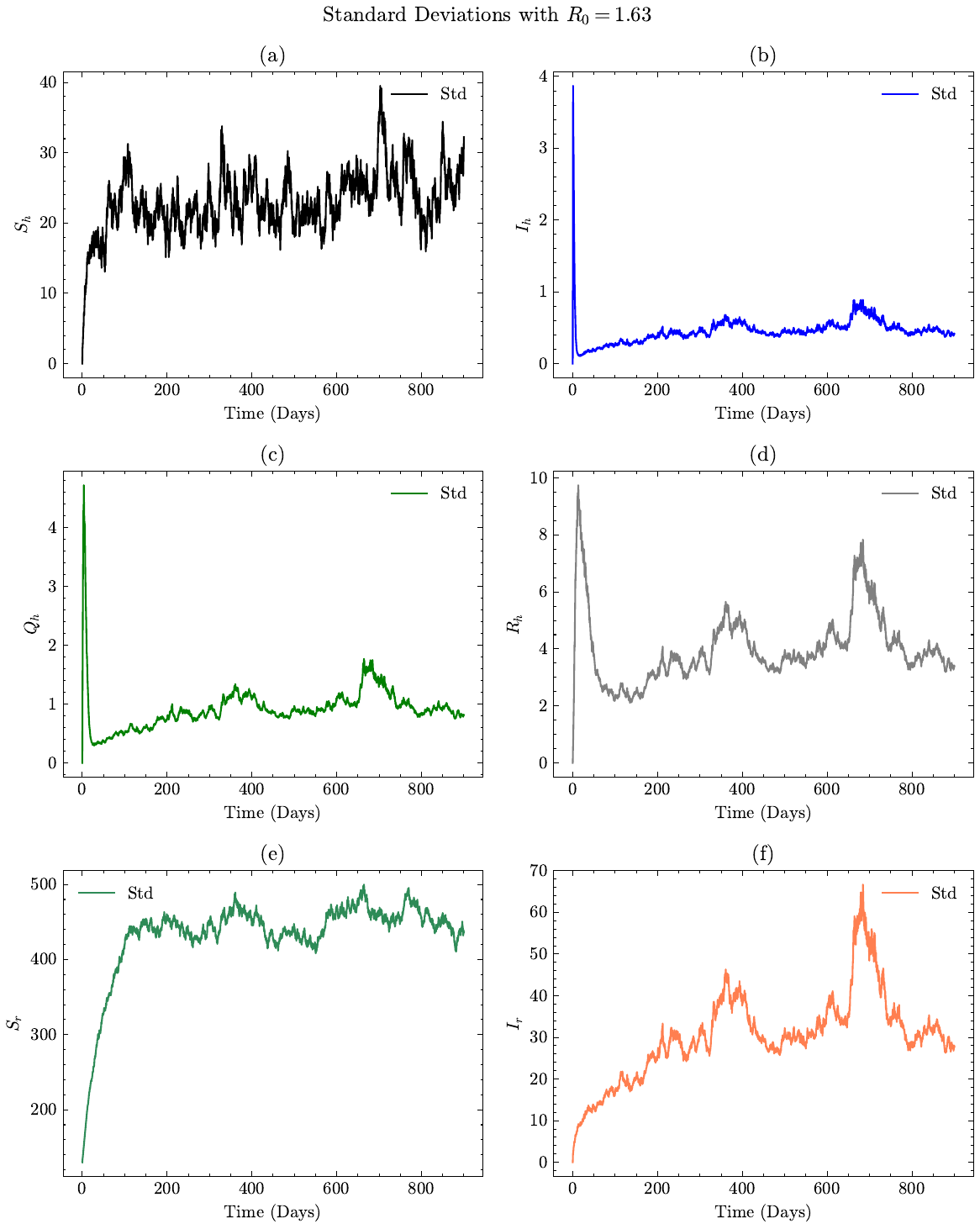}}
\caption{For $R_0>1$, curves show the variability of each compartment humans (a)-(d) and rodents (e)-(f) over multiple simulations, indicating the uncertainty associated with each curve.}
\end{figure}

\section{Speed of growth of time-depending stochastic  model for monkey-pox }
\noindent Several generalized SIR (Susceptible, Infected, and Recovered) models have been studed in terms of their  long-term behaviour. For instance,  the SIR model with perturbed disease transmission coefficients of \cite{Dieu} has the  rates of convergence  determined for $t\to \infty$. It is shown that the rate is not too far from exponential in the sense that the convergence speed has a polynomial form of any degree. In \cite{Jiang}, \cite{Lin2}, the long-time behaviour of densities of the distributions of the stochastic SIR  model are studied.

In this section, we focus on the influence  of  the various system parameters on the asymptotic behaviour of the non-autonomous stochastic model, which corresponds system (\ref{STD1.1}). For this study, the parameters of the model are allowed to depend on time, to incorporate the seasonal variation. Hence the model is of the form 
 \begin{align}\label{SDE_t}
    dS_h(t)&= \left(\theta_h(t)-(1-p)\left(\frac{\eta_1(t)I_r(t)+\eta_2(t)I_h(t)}{N_h(t)}\right)S_h(t)-\mu_h(t)S_h(t)\right)dt\nonumber\\
    & -(1-p)\left(\frac{\sigma_1(t)I_r(t)dB_1(t)+\sigma_2(t)I_h(t)dB_2(t)}{N_h(t)}\right)S_h(t)+\sigma_3(t)S_h(t)dB_3(t),\nonumber\\
    dI_h(t)&=\left((1-p)\left(\frac{\eta_1(t)I_r(t)+\eta_2(t)I_h(t)}{N_h(t)}\right)S_h(t)-(\mu_h(t)+\delta_h(t)+\zeta(t))I_h(t)\right)dt\nonumber\\
    &+(1-p)\frac{\sigma_2(t)S_h(t)I_h(t)}{N_h(t)}dB_2(t)+\sigma_4(t)I_h(t)dB_4(t),\nonumber\\
   dQ_h(t)&=\left(\zeta(t) I_h(t)-(\mu_h(t)+\gamma_h(t)+(1-\theta(t))\delta_h(t))Q_h(t)\right)dt+\sigma_5(t)Q_h(t)dB_5(t),\nonumber\\
    dR_h(t)&=\left(\gamma_h(t)Q_h(t)-\mu_h(t)R_h(t)\right)dt+\sigma_6(t)R_h(t)dB_6(t)\nonumber\\
    dS_r(t)&=\left(\theta_r(t)-\frac{\eta_3(t)S_r(t)I_r(t)}{N_r(t)}-\mu_r(t)S_r(t)\right)dt+\sigma_7(t)S_r(t)dB_7,\nonumber \\
dI_r(t)&=\left(\frac{\eta_3(t)S_r(t)I_r(t)}{N_r(t)}-(\mu_r(t)+\delta_r(t))I_r(t)\right)dt+\sigma_8(t)I_r(t)dB_8(t).
\end{align}

Our goal is to determine which parameters have a significant effect on the growth rate of the solution of this non-autonomous system. We follow similar arguments to  \cite{bul, tym} to achieve the results.

 We stress that determining the rate of growth is essential for establishing the severity of the outbreak. In this context, we consider the model (\ref{STD1.1}) with time-dependent parameters leading to (\ref{SDE_t}). This makes it possible to take into account, e.g., seasonal changes. 
We emphasize that the rate of growth provides important insights into the dynamic nature of the disease, indicating whether it is accelerating or decelerating.

We can still denote the solution of the time-depending stochastic epidemic model  (\ref{SDE_t}),  as $$\mathcal{X}(t)=(S_h(t),I_h(t),Q_h(t),R_h(t),S_r(t),I_r(t)),$$ given an initial condition of $\mathcal{X}(0)\in \mathbb{R}^6_+$ on $t \geq 0$. The following theorem
establish the well-posedness of the system (\ref{SDE_t}).
\begin{theorem}
 Let the coefficients of the stochastic system (\ref{SDE_t}) be continuous functions with upper bounds. Then for every initial condition $\mathcal{X}(0)\in \mathbb{R}^6_+$, the stochastic system (\ref{SDE_t}) admit a unique, a.s. non-negative solution.
\end{theorem}
 \begin{proof} The proof is based on  similar arguments as in the proof of  Theorem \ref{Global}.
 \end{proof}
\noindent To  achieve our goals, we first study the conditions under which the terms involving differentials of Brownian motion vanishes as $t\to \infty$.
 To do this, let us consider a function $F:\R^6_+\to \R_+$. Assume that $F :=F\left(\sum_{i=1}^6x_i\right) $ is an increasing twice continuously differentiable function:  

$1^{\circ}$. $F'=F'_{x_1}=F'_{x_2}=F'_{x_3}=F'_{x_4}=F'_{x_5}=F'_{x_6}$;

$2^{\circ}$. $F''=F''_{x_1x_1}=F''_{x_2x_2}=F''_{x_3x_3}=F''_{x_4x_4}=F''_{x_5x_5}=F''_{x_6x_6}$,

with $F'_{x_k}=\frac{\partial F\left(\sum_{i=1}^6x_i\right)}{\partial x_k}$ and $F''_{x_kx_k}=\frac{\partial^2 F\left(\sum_{i=1}^6x_i\right)}{\partial {x_k}^2}$  for $k=1,...,6$.

\noindent Let $\sigma_j,$ $j=3,...,8$ be continuous positive functions. Denote
$$\mathcal{I}(t)=\int_0^t F'(\mathcal{K}(s))\cdot\left[\sigma_3(s)S_h(s)dB_3(s)+\sigma_4(s)I_h(s)dB_4(s)+\sigma_5(s)Q_h(s)dB_5(s)\right.$$
$$\left.+\sigma_6(s)R_h(s)dB_6(s)+\sigma_7(s)S_r(s)dB_7(s))+\sigma_8(s)I_r(s)dB_8(s)\right],$$
where $\mathcal{K}$ such as in (\ref{K}), and we get the following result. 

\begin{theorem}\label{ti} Let $\left\{ B_j(t): t\geq0   \right\}, j=3,4,5,6,7,8,$ be independent Brownian processes, $\sigma_j,$ $j=3,...,8$ be continuous positive functions, and for  $i=1,...,6, $ function $x_iF'\left(\sum_{i=1}^6x_i\right)$, $x=(x_1,..., x_6)\in \R^6, $ be  bounded.  Assume 
\begin{equation}\label{Int_s}
  \sum _{n=1}^\infty \frac{1}{2^{2n}}\int_0^{2^{n+1}} \Xi(s) \,ds<\infty,  
\end{equation} 
where $\Xi(t)= max\{\sigma_j(t),\, t\geq0,\, j=3,4,5,6,7,8\}$. 
Then
\begin{equation}\label{dopt}
\underset{ t\to \infty}{\lim} \frac {\mathcal{I}(t)}{t} =0 \,\,\,\textit{a.s.}
\end{equation}
\end{theorem}
\begin{proof}
Let us fix $\epsilon> 0$. Using the Doob's inequality
\begin{eqnarray}  &&\mathrm{P}\left\{ \underset{k\geq n}{\sup} \frac{|\mathcal{I}(2^ k)|}{2^{ k}} >\epsilon \right\}\leq \mathrm{P}\left\{ \underset{t\geq 2^n}{\sup} \frac {|\mathcal{I}(t)|}{t} >\epsilon \right\}\leq
 \sum_{k=n}^{\infty} \mathrm{P}\left\{ \underset{2^{k}\leq t\leq 2^{k+1}}{\sup} \frac {|\mathcal{I}(t)|}{t} >\epsilon \right\}\cr\cr&&
\leq \sum_{k=n}^{\infty} \mathrm{P}\left\{ \underset{2^{k}\leq t\leq 2^{k+1}}{\sup} \left(\frac {1}{t} \Big|\int_0^{t}F'(\mathcal{K}(s))\sigma_3(s)S_h(s)dB_3(s)\Big|\right)>\frac{\epsilon}{6} \right\}\cr\cr&&
+\sum_{k=n}^{\infty} \mathrm{P}\left\{ \underset{2^{k}\leq t\leq 2^{k+1}}{\sup} \left(\frac {1}{t} \Big|\int_0^{t} F'(\mathcal{K}(s))\sigma_4(s)I_h(s)dB_4(s) \Big|\right)>\frac{\epsilon}{6} \right\}\cr\cr&&+\sum_{k=n}^{\infty} \mathrm{P}\left\{ \underset{2^{k}\leq t\leq 2^{k+1}}{\sup} \left(\frac {1}{t} \Big|\int_0^{t} F'(\mathcal{K}(s))\sigma_5(s)Q_h(s)dB_5(s) \Big|\right)>\frac{\epsilon}{6} \right\}\cr\cr&&+\sum_{k=n}^{\infty} \mathrm{P}\left\{ \underset{2^{k}\leq t\leq 2^{k+1}}{\sup} \left(\frac {1}{t} \Big|\int_0^{t} F'(\mathcal{K}(s))\sigma_6(s)R_h(s)dB_6(s) \Big|\right)>\frac{\epsilon}{6} \right\}\cr\cr&&+\sum_{k=n}^{\infty} \mathrm{P}\left\{\underset{2^{k}\leq t\leq 2^{k+1}}{\sup} \left(\frac {1}{t} \Big|\int_0^{t} F'(\mathcal{K}(s))\sigma_7(s)S_r(s)dB_7(s)ds \Big|\right)>\frac{\epsilon}{6} \right\}\cr\cr&&+\sum_{k=n}^{\infty} \mathrm{P}\left\{\underset{2^{k}\leq t\leq 2^{k+1}}{\sup} \left(\frac {1}{t} \Big|\int_0^{t} F'(\mathcal{K}(s))\sigma_8(s)I_r(s)dB_8(s) \Big|\right)>\frac{\epsilon}{6} \right\}\cr\cr&&\leq \frac{144}{\epsilon^2} \sum _{k=n}^\infty \frac{1}{2^{2 k}}\int_0^{2^{k+1}}\mathrm{E} \Big(F'(\mathcal{K}(s))^2\cdot(\sigma^2_3(s)S^2_h(s)+\sigma^2_4(s)I^2_h(s)+\sigma^2_5(s)Q^2_h(s)\cr\cr&&+\sigma^2_6(s)R^2_h(s)+\sigma^2_7S^2_r(s)+\sigma^2_8I^2_r(s)\Big)ds\leq\cr\cr&&\leq \frac{144}{\epsilon^2} \sum _{k=n}^\infty \frac{1}{2^{2 k}}\int_0^{2^{k+1}} \Xi (s)\cdot\mathrm{E}\Big(F'(\mathcal{K}(s))^2\cdot(S^2_h(s)+I^2_h(s)+Q^2_h(s)\cr\cr&&+R^2_h(s)+S^2_r(s)+I^2_r(s)\Big)ds\leq\frac{144}{\epsilon^2} \sum _{k=n}^\infty \frac{1}{2^{2 k}}\int_0^{2^{k+1}} \Xi (s)\cdot 6\tilde{C}^2_1 ds\cr\cr&&=\frac{864\cdot \tilde{C}^2_1}{\epsilon^2}  \sum _{k=n}^\infty \frac{1}{2^{2 k}}\int_0^{2^{k+1}} \Xi (s)ds,\end{eqnarray}
where $\tilde{C}_1=\underset{x_i}{\sup}\left\{x_iF'\left(\sum_{i=1}^6x_i\right)\right\}$.
Thus 
$$\mathrm{P}\left\{ \underset{k\geq n}{\sup} \frac{ \Big|\mathcal{I}(2^ k) \Big|}{2^{ k}} >\epsilon \right\}\leq \frac{864\cdot \tilde{C}^2_1}{\epsilon^2}  \sum _{k=n}^\infty \frac{1}{2^{2 k}}\int_0^{2^{k+1}} \Xi (s)ds$$
for any $n\geq 0$ and $\epsilon>0$.
\noindent By assumption, we have  that   $$\displaystyle  \sum _{k=n}^\infty \frac{1}{2^{2 k}}\int_0^{2^{k+1}} \Xi(s) ds\to \,\,0, \,\, \textnormal{as} \,\, n\to\infty.$$ 
Since $\epsilon>0$ is arbitrary, we have the relations
$$\mathrm{P}\left\{ \underset{ t\to \infty}{\lim \sup} \frac { \Big|\mathcal{I}(t) \Big|}{t} >0 \right\}=0\Rightarrow \mathrm{P}\left\{ \underset{ t\to \infty}{\lim \sup} \frac { \Big|\mathcal{I}(t) \Big|}{t} =0 \right\}=1.$$
\[\]
Therefore,
$$\underset{ t\to \infty}{\lim} \frac { \Big|\mathcal{I}(t) \Big|}{t} =0 \,\,\,\, \textnormal{a.s.}$$
i.e. relation (\ref{dopt}) holds, and Theorem \ref{ti} is proved.
\end{proof}
\begin{remark}
    For the monkeypox model (\ref{STD1.1}) wich has constant volatilities  $\sigma_i,  i=3,4,5,6,7,8,$ the conditions of the Theorem \ref{ti} are fulfilled automatically. Indeed, if $m_{\sigma}=max\{\sigma_3, \sigma_4, \sigma_5, \sigma_6, \sigma_7, \sigma_8\},$ then 
    $$ \sum _{n=1}^\infty \frac{1}{2^{2n}}\int_0^{2^{n+1}} \Xi(s)\, ds= \sum _{n=1}^\infty \frac{1}{2^{2n}}\int_0^{2^{n+1}} m_{\sigma}\, ds=m_{\sigma}\sum _{n=1}^\infty \frac{1}{2^{n-1}}<\infty.$$
   Hence, we obtain (\ref{dopt}) from  the properties of Brownian motion.
   \end{remark}

\noindent Now let's move to the study of  the extinction of the disease. Consider the operator $L$ which acts on the Lyapunov function $\mathcal{V}:\R^6_+\to \R_+$ given by $\mathcal{V}(\mathcal{X}(t))=F(\mathcal{K}(t))$ have the form
\begin{eqnarray*}
&&L\mathcal{V}(\mathcal{X}(t))=F'(\mathcal{K}(t))\bigg(\theta_h(t)-\mu_h(t)(S_h(t)+I_h(t)+Q_h(t)+R_h(t))\cr\cr&&-\delta_h(t)I_h(t)-(1-\theta(t))\delta_h(t)Q_h(t)+ \theta_r(t)-\mu_r(t)(S_r(t)+I_r(t))-\delta_r(t)I_r(t)\bigg)\cr\cr&&+  
F''(\mathcal{K}(t))\bigg(\sigma^2_3(t)S^2_h(t)+\sigma^2_4(t)I^2_h(t)+\sigma^2_5(t)Q^2_h(t)+\sigma^2_6(t)R^2_h(t)\cr\cr&&+\sigma^2_7(t)S^2_r(t)+\sigma^2_8(t)I^2_h(t)\bigg)\cr\cr&&+\frac{(1-p)^2\sigma^2_1(t)I^2_r(t)S^2_h(t)F''(\mathcal{K}(t))}{N^2_h(t)}+\frac{2(1-p)^2\sigma^2_2(t)I^2_h(t)S^2_h(t)F''(\mathcal{K}(t))}{N^2_h(t)}.
\end{eqnarray*}
The following theorem describes the  parameters that mostly affect the rate of the disease spread.
\begin{theorem}\label{5.3}
    Assume that  all conditions Theorem \ref{ti} hold and suppose
    
    I. $\underset{t\to \infty} \lim \frac{1}{t}\int_0^t\Big|\mathcal{V}'(\mathcal{X}(s))\Big|(\theta_h(s)+ \theta_r(s))ds =a\in[0;+\infty)$ a.s.;
    
    II. $\underset{ t\to \infty}{\lim}\frac{1}{t}\int_0^tH_1(s)ds=b\in(0;+\infty)$ with $H_1(t)=\mu_h(t)\vee\mu_r(t)$; 

III. $\underset{ t\to \infty}{\lim}\frac{1}{t}\int_0^t\left((2-\theta(s))\delta_h(s)+\delta_r(s)\right)ds=c\in[0;+\infty).$

    IV. $\lim\underset{ t\to \infty}{\lim}\frac{1}{t}\int_0^t(\sigma^2_1(s)+2\sigma^2_2(s))ds=d\in[0;+\infty)$;
Then
$$\underset{ t\to \infty}{\lim} \frac {\mathcal{V}(\mathcal{X}(t))}{t}\leq a+3\tilde{C}_1\cdot(2 b+c)+(1-p)^2\cdot \tilde{C}_2\cdot d  \,\,\ \textit{a.s.},$$
where $\tilde{C}_1=\sup\left\{\mathcal{V}' x_1\right\}$, $\tilde{C}_2=\sup\{\mathcal{V}''x^2_1\}$.
   
\end{theorem}
\begin{proof}
Applying It\^o formula to $\mathcal{V},$   we obtain
 \begin{eqnarray*}  
d(\mathcal{V}(\mathcal{X}(t))) &= & L \mathcal{V}(\mathcal{X}(t))dt+\mathcal{V}'(\mathcal{X}(t))\bigg(\frac{(p-1)\sigma_1I_r(t)S_h(t)}{N_h(t)}dB_1(t)\cr\cr &&+\sigma_3(t)S_h(t)dB_3(t)+\sigma_4(t)I_h(t)dB_4(t)+\sigma_5(t)Q_h(t)dB_5(t)\cr\cr && +\sigma_6(t)R_h(t)dB_6(t)+\sigma_7(t)S_r(t)dB_7(t))+\sigma_8(t)I_r(t)dB_8(t)\bigg),\end{eqnarray*}
or in integral form
$$\mathcal{V}(\mathcal{X}(t))=\mathcal{V}(\mathcal{X}(0))+\int_0^t L \mathcal{V}(\mathcal{X}(s))ds+(p-1)\sigma_1\int_0^t\frac{\mathcal{V}'(\mathcal{X}(s))I_r(s)S_h(s)}{N_h(s)}dB_1(s) +\mathcal{I}(t).$$
By Theorem \ref{ti}, we get 
$$\underset{ t\to \infty}{\lim} \frac {\mathcal{I}(t)}{t} =0 \,\,\,
\textit{a.s.}$$
Since $S_h<N_h$, and $x_i\mathcal{V}'$, $i=1,...,6, $ is bounded, then  the strong low of large number  gives
$$\underset{ t\to \infty}{\lim} \frac{(p-1)\sigma_1}{t}\int_0^t\frac{\mathcal{V}'(\mathcal{X}(s))I_r(s)S_h(s)}{N_h(s)}dB_1(s)= (p-1)\sigma_1\tilde{C}_1\underset{ t\to \infty}{\lim} \frac{B_1(t)}{t}=0.$$
While (\ref{Int_s}) holds, we have
 \begin{eqnarray*}  &&\mathcal{I}_1(t)=\int_0^t\mathcal{V}''(\mathcal{X}(s))(\sigma^2_3(s)S^2_h(s)+\sigma^2_4(s)I^2_h(s)+\sigma^2_5(s)Q^2_h(s) \cr\cr &&+\sigma^2_6(s)R^2_h(s)+\sigma^2_7(s)S^2_r(s)+\sigma^2_8(s)I^2_r(s))ds\cr\cr &&\leq \int_0^t\Xi(s)\cdot F''(\mathcal{K}(s))(S^2_h(s)+I^2_h(s)+Q^2_h(s)+R^2_h(s)+S^2_r(s)+I^2_r(s))ds\cr\cr && \leq 6\tilde{C}_2\int_0^t\Xi(s)ds. \end{eqnarray*}

\noindent Therefore
$\displaystyle \underset{ t\to \infty}{\lim}\frac{\mathcal{I}_1(t)}{t}=0 \,\,\textnormal{a.s.} $ Since $\dfrac{S^2_h(s)\mathcal{V}''(\mathcal{X}(s))}{N^2_h(s)}\leq \mathcal{V}''(\mathcal{X}(s))$. It is easy to see, that
\begin{eqnarray*}  &&\mathcal{I}_2(t)=(1-p)^2\int_0^t(\sigma^2_1(s)I^2_r(s)+2\sigma^2_2(s)I^2_h(s))\frac{S^2_h(s)\mathcal{V}''(\mathcal{X}(s))}{N^2_h(s)}ds\cr\cr &&\leq (1-p)^2\cdot \tilde{C}_2\int_0^t(\sigma^2_1(s)+2\sigma^2_2(s))ds.
\end{eqnarray*}
Next estimation, 
 \begin{eqnarray*}  &&
\mathcal{I}_3(t)=\Big|\int_0^t\mathcal{V}'(\mathcal{X}(s))\bigg(\theta_h(s)+ \theta_r(s)-\mu_h(s)(S_h(s)+I_h(s)+Q_h(s)+R_h(s))\cr\cr&&-\mu_r(s)(S_r(s)+I_r(s))-\delta_h(s)I_h(s)-(1-\theta(s))\delta_h(s)Q_h(s)-\delta_r(s)I_r(s)\bigg)ds\Big|\cr\cr&&\leq \int_0^t\Big|\mathcal{V}'(\mathcal{X}(s))\Big|(\theta_h(s)+ \theta_r(s))ds+6\tilde{C}_1\int_0^tH_1(s)ds\cr\cr&&+3\tilde{C}_1\int_0^t\left(\delta_h(s)+(1-\theta(s))\delta_h(s)+\delta_r(s)\right)ds.
\end{eqnarray*}
Here $H_1(t)=\mu_h(t)\vee\mu_r(t)$. Finally,
\begin{eqnarray*}  
 \underset{ t\to \infty}{\lim} \frac {\mathcal{V}(\mathcal{X}(t))}{t}&=&\underset{ t\to \infty}{\lim} \frac {\mathcal{V}(\mathcal{X}(0))}{t}+\underset{ t\to \infty}{\lim} \frac {1}{t}\int_0^t L \mathcal{V}(\mathcal{X}(s))ds\cr\cr&& +\underset{ t\to \infty}{\lim} \frac{(p-1)\sigma_1}{t}\int_0^t\frac{\mathcal{V}'(\mathcal{X}(s))I_r(s)S_h(s)}{N_h(s)}dB_1(s)+\underset{ t\to \infty}{\lim} \frac { \mathcal{I}(t)}{t}\cr\cr&& = \underset{ t\to \infty}{\lim}\frac{1}{t}\int_0^t\Big|\mathcal{V}'(\mathcal{X}(s))\Big|(\theta_h(s)+ \theta_r(s))ds+\underset{ t\to \infty}{\lim}\frac{6\tilde{C}_1}{t}\int_0^tH_1(s)ds\cr\cr&&+(1-p)^2\tilde{C}_2\cdot\underset{ t\to \infty}{\lim}\frac{1}{t}\int_0^t(\sigma^2_1(s)+2\sigma^2_2(s))ds\cr\cr&&+3\tilde{C}_1\cdot\underset{ t\to \infty}{\lim}\frac{1}{t}\int_0^t\left(\delta_h(s)+(1-\theta(s))\delta_h(s)+\delta_r(s)\right)ds\cr\cr&&\leq a+6\tilde{C}_1\cdot b+(1-p)^2\cdot \tilde{C}_2\cdot d+3\tilde{C}_1 \cdot c.\end{eqnarray*}
Thus, Theorem \ref{5.3} is proved.
\end{proof}
\begin{remark} The verification of assumption I. is quite tedious, as it depends on the stochastic processes. Therefore, it is advisable to use the following sufficient conditions:
$\underset{x} \sup\Big|\mathcal{V}'(x)\Big|<\infty$ and $\underset{ t\to \infty}{\lim} \frac{1}{t}\int_0^t(\theta_h(s)+ \theta_r(s))ds =A\in[0;+\infty)$ a.s.
\end{remark}
\begin{corollary}\label{5_5}Let $I_h, Q_h, I_r$ satisfy the system (\ref{SDE_t}).  Assume 
\begin{equation*}\label{Int_s-1}
  \sum _{n=1}^\infty \frac{1}{2^{2n}}\int_0^{2^{n+1}} \tilde{\Xi}(s) ds<\infty,  
\end{equation*} 
where $\tilde{\Xi}(t)= max\{\sigma_i(t), t\geq0, i=2,4,5,6,\}$. If 
\begin{equation}\label{Int_s-2}\underset{t\to\infty}\lim\sup\frac{1}{t}\int_0^t\left((1-p)(\eta_1(s)+\eta_2(s))+\eta_3(s)\right)ds<\underset{t\to\infty}\lim\sup\frac{1}{t}\int_0^t(M_1(s)+M_2(s))ds,\end{equation}
for $M_1(t)= \min\{\mu_h(t),\mu_r(t)\}$ and $M_2(t)=\min\{\delta_h(t),\delta_r(t)\}$, $t\geq0$, then the classes $I_h, Q_h, I_r$ will die out exponentially with probability one, i.e.,
\begin{eqnarray*}  
 \underset{ t\to \infty}{\lim \sup} \frac {\ln (I_h(t)+Q_h(t)+I_r(t))}{t}&<&0\,\, \textnormal{a.s.}
 \end{eqnarray*}
\end{corollary}
\begin{proof} The proof follows from the It\^o formula for $\ln (I_h+Q_h+I_r)$ and  Theorem \ref{ti}.
\end{proof}
\begin{remark}The result of Corollary \ref{5_5}  generalizes the result of Theorem \ref{Exti} for the monkeypox model (\ref{STD1.1}) where parameters are in the form of constants. It should be noted that assumption (\ref{Int_s-2}) corresponds to $\mathcal{R}_0<1$ in Theorem \ref{Exti} in cases of constant. 
\end{remark}
The rate of expansion shows how rapidly the monkeypox is spreading. The growth rate over time is a good indicator of the severity and speed of the epidemic.
 A faster rate of increase indicates that the population's health is declining more quickly, putting more strain on the healthcare systems.

%  %%%%%%%%%%%%%%%%%%%%%%%%%%%%%%%%%%%%%%%%%%%%%%%%%
\section{Conclusion and Final Comments}
\noindent In the present research work, we have introduced and  analyzed a  stochastic monkeypox epidemic model with
cross-infection between animals and humans. Our results discuss the well-posedness of the model to secure     that the simulations and prediction will be meaningfull. Specifically, we have seen
global existence, positivity and uniqueness of solution. 
Subsequently, we have expounded results for stochastic ultimate boundedness and permanence and results related to the long-run behavior via exploring asymptotic properties of the model in a time-depended variation. 

By anticipating and preparing for possible outbreaks, public health professionals may utilize these forecasts to lessen the burden of the disease with early and focused actions. The compartment modeling technique can incorporate compartments for immune people to better understand the fading and development of immunity over time and the immunity of the population. It may also be used to find possible amplification hosts, such rodents in the case of monkeypox, and reservoirs.
This information is critical for understanding the ecology of the disease, directing monitoring efforts, and putting control measures in place. Randomness in the spread of disease given by the uncertain environment is taken into account. This facilitates the creation of probabilistic forecasts, increases prediction reliability, and directs decision-making in the face of uncertainty.

\section*{Statements and Declarations}
\noindent \textbf{Funding.} The present research is carried out within the frame and support of the ToppForsk project nr. 274410 of the Research Council of Norway with the title STORM: Stochastics for Time-Space Risk Models, and the project "Dynamics for Transmission
of Monkeypox: Biological and Probabilistic
Behaviour" supported by the Higher Education Commission of Pakistan, and supported by the MSCA4Ukraine grant (AvH ID:1233636), which is funded by the European Commission.

\noindent \textbf{Acknowledgements.} The first author would like to thank the University of Swat, KP Pakistan, for providing the opportunity to complete nine months postdoc with stay at the University of Oslo, Norway, and CIMPA, Nice, France to have provided the grounds to meet and form the present scientific collaboration.
%%%%%%%%%%%%%%%%%%%%%%%%%%%%
%%%%%%%%%%%%%%%%%%%%%%%%%%%%%%%%%%%

\end{document}